\documentclass[11pt, a4paper]{article}

\usepackage{inputenc} 
\usepackage[english]{babel}
\usepackage{dsfont}

\usepackage[totalheight=24 true cm, totalwidth=17 true cm]{geometry}
\usepackage{enumitem}

\setlist[itemize]{noitemsep}

\usepackage{bbding}
\usepackage{amsmath,multirow}
\usepackage{amsthm}
\usepackage{amssymb}
\usepackage{subcaption}
\usepackage{mathrsfs}
\usepackage{mathtools}
\usepackage{ stmaryrd }
\usepackage{url}
\usepackage{wrapfig}
\usepackage{starfont}
\usepackage{pifont}
\usepackage{eurosym}

\usepackage{imakeidx}
\makeindex[]

\usepackage{multicol}
\usepackage{relsize}
\usepackage{float}
\usepackage{tikz,pgf}
\usepackage{tikz-cd}
\usepackage{wasysym}
\usepackage{graphicx}
\usepackage{fancyhdr}
\usepackage{verbatim}

\usepackage{pgfplots}
\pgfplotsset{compat=1.15}
\usepackage{mathrsfs}
\usetikzlibrary{arrows}

\usepackage[all,dvips,knot,web,arc,curve,color,frame]{xy}
\usepackage[all,knot,arc,color,web]{xy}
\xyoption{arc}
\xyoption{web}
\xyoption{curve}

\newtheorem{theoremA}{Theorem}

\newtheorem{theoremB}{Theorem}

\numberwithin{equation}{section}

\theoremstyle{plain}
\newtheorem{theorem}{Theorem}[section]
\newtheorem{proposition}[theorem]{Proposition}
\newtheorem{corollary}[theorem]{Corollary}
\newtheorem{lemma}[theorem]{Lemma}

\theoremstyle{definition}
\newtheorem{definition}[theorem]{Definition}

\newtheorem{remark}[theorem]{Remark}

\newtheorem{notation}[theorem]{Notation}
\newtheorem{example}[theorem]{Example}
\newtheorem{question}[theorem]{Question}
\usepackage[bookmarksnumbered=true]{hyperref} 
\hypersetup{
     colorlinks = true,
     linkcolor = blue,
     anchorcolor = blue,
     citecolor = teal,
     filecolor = blue,
     urlcolor = blue
     }
\newcommand\restr[2]{{
  \left.\kern-\nulldelimiterspace 
  #1 
  \vphantom{\small|} 
  \right|_{#2} 
  }}

\everymath{\displaystyle}
\allowdisplaybreaks

\makeatletter
\def\mathcenterto#1#2{\mathclap{\phantom{#1}\mathclap{#2}}\phantom{#1}}
\let\old@widetilde\widetilde
\def\widetildeto#1#2{\mathcenterto{#2}{\old@widetilde{\mathcenterto{#1}{#2\,}}}}
\let\old@widehat\widehat
\def\widehatto#1#2{\mathcenterto{#2}{\old@widehat{\mathcenterto{#1}{#2\,}}}}
\makeatother

\newtheorem{algorithm}[theorem]{Algorithm}

\newcommand{\ip}[2]{\langle  #1,#2 \rangle} 
\newcommand{\is}[1]{\langle  #1 \rangle} 

\newcommand{\size}[1]{\left| #1 \right|} 
\newcommand{\pare}[1]{\left( #1 \right)} 
\newcommand{\set}[1]{{\left\{ #1 \right\}}} 

\newcommand{\Li}{{\mathfrak{S}}}
\newcommand{\corch}[1]{\left[ #1 \right]} 

\newcommand*\closure[1]{\overline{#1}}

\DeclareMathOperator{\rank}{rk}

\DeclareMathOperator{\im}{im}
\def\dim{\operatorname{dim}}

\DeclareMathOperator{\CC}{\mathbb{C}}

\hypersetup{
    colorlinks=true,
    linkcolor=blue,
    filecolor=magenta,      
    urlcolor=cyan,
    pdftitle={Overleaf Example},
    pdfpagemode=FullScreen,
    }
    
\def\Id{\operatorname{Id}}

\setlist[itemize]{noitemsep}

\normalfont

\usepackage{bbm}
\usepackage{dsfont}

\usepackage{url}
\usepackage{amsmath,blkarray,booktabs, bigstrut}
\usepackage{tikz}
\usepackage{multirow}
\usepackage{xcolor}
\usepackage{color}
\usepackage[maxbibnames=99]{biblatex}

\renewbibmacro{in:}{}

\title{Solvable and Nilpotent Matroids: Realizability and Irreducible Decomposition of Their Associated Varieties}
\author{Emiliano Liwski and Fatemeh Mohammadi}
\date{}

\begin{document}
\maketitle

\begin{abstract}
We introduce the families of solvable and nilpotent matroids, examining their realization spaces, closures, and associated matroid and circuit varieties. We study their realizability, as well as the irreducible decomposition 
of their associated matroid and circuit varieties. Additionally, we describe 
a finite generating set for the corresponding ideals, considered up to radical.
We establish sufficient conditions for both the realizability of these matroids and the irreducibility of their associated varieties. Specifically, we establish the realizability and irreducibility of matroid varieties associated with nilpotent matroids and prove the irreducibility of matroid varieties arising from certain classes of solvable paving matroids.
Additionally, we analyze the defining polynomial equations of these varieties using Grassmann-Cayley algebra and geometric liftability techniques. Furthermore, we provide a complete generating set for the 
matroid ideals associated with forest configurations.
\end{abstract}

{\hypersetup{linkcolor=black}
{\tableofcontents}}

\vspace{-2mm
}

\section{Introduction}
Matroid realization spaces have long been a central research focus in algebraic combinatorics and algebraic topology, providing an extensive source of
singular spaces in algebraic geometry. The Mnëv-Sturmfels universality theorem demonstrates that matroid realization spaces adhere to the Murphy's law in algebraic geometry, meaning that for each singularity type, there exists a matroid of rank $3$
whose realization space exhibits that singularity; see e.g.~\cite{mnev1985manifolds, mnev1988universality,lee2013mnev}.

Moreover, there is a natural correspondence between the realization spaces of matroids and those of hyperplane arrangements. Each collection of vectors realizing a matroid $M$ corresponds to a hyperplane arrangement in the dual vector space; \cite{bokowski1989computational,richter1999universality}. Thus, the matroid realization space serves as a parameter space for hyperplane arrangements with fixed combinatorial structure. Furthermore, hyperplane arrangements within the same connected component of the realization space have diffeomorphic complements \cite{randell1989lattice}. A natural question is to determine when the 
realization space is irreducible. In general, this does not hold, as demonstrated by Rybnikov in \cite{rybnikov2011fundamental}, where he constructed a counterexample consisting of two arrangements with the same combinatorial structure whose complements have non-isomorphic fundamental groups.

\smallskip

Our primary focus on matroid realization spaces is centered on the challenge of determining their irreducible decomposition. As noted in \cite{knutson2013positroid}, this is generally a difficult problem, and the existing literature offers limited insights. In \cite{nazir2012connectivity}, it was shown that the realization spaces of inductively connected line arrangements are irreducible. This result was extended in \cite{Emiliano-Fatemeh} to higher-rank matroids, specifically inductively connected hyperplane arrangements, where it was shown that their realization spaces are irreducible. 

\smallskip

In this paper, we introduce the families of solvable and nilpotent matroids,
and
study the irreducible decomposition of their realization spaces. The family of solvable matroids is quite extensive, encompassing nearly all examples discussed in the literature and serving as a natural generalization of inductively connected line arrangements \cite{nazir2012connectivity} to higher dimensions. On the other hand, 
nilpotent matroids are defined as a natural extension of forest configurations \cite{clarke2021matroid} to higher dimensions.

\smallskip

Matroid varieties were first introduced in \cite{gelfand1987combinatorial} and have since been widely studied; see~e.g.~\cite{clarke2021matroid,sidman2021geometric,Fatemeh3}. 
These varieties  
naturally extend the notion of matroid realization spaces, by defining them as their Zariski closure. These varieties
play a fundamental role in the study of determinantal varieties in commutative algebra, \cite{bruns2003determinantal,clarke2020conditional,clarke2021matroid,herzog2010binomial,pfister2019primary}, and have
applications in conditional independence (CI) models in algebraic statistics, \cite{clarke2021matroid,clarke2020conditional,clarke2022conditional,caines2022lattice, mohammadi2018prime, ene2013determinantal}. The 
central challenge regarding matroid varieties is to compute their defining equations, or equivalently, the ideals of the matroid varieties, $I_{M}$, and to 
provide a combinatorial interpretation of these polynomials. The classical method for computing such polynomials is  via the Grassmann-Cayley algebra; see e.g.~\cite{computationalgorithms,sidman2021geometric}. Recently, in \cite{Fatemeh4}, a new
method for constructing such polynomials was given, relying on the liftability technique; see also~\cite{richter2011perspectives}. 

Determining a generating set for a general matroid ideal $I_{M}$, up to radical, remains a challenging task. Existing algorithms for computing these generators have high complexity, thus restricting their applicability to very small matroids. This difficulty emerges  
from the saturation step  
involved in the process, see e.g.~\cite{white2017geometric, sturmfels1989matroid, Stu4, sitharam2017handbook}. Additionally, Gr\"obner basis techniques 
also become extremely difficult when the matroid does not belong to a particularly well-behaved family. 
Due to the difficulty involved in calculating a generating set, up to radical, for $I_M$, 
considerable effort has been focused on producing specific polynomials within this ideal, even if  
a complete generating set is not achieved,  
especially for certain classes of matroids; see e.g.~\cite{sturmfels1989matroid,sidman2021geometric,white2017geometric, sturmfels1989matroid, Stu4, sitharam2017handbook}. In this work, we compute a finite generating set, up to radical, for the matroid ideals of forest 
configurations 
and construct polynomials within the matroid ideals of nilpotent matroids.

\smallskip
\noindent{\bf Our contributions.}
We  
now summarize the main results of this paper. In Definitions~\ref{k2} and~\ref{m7}, we introduce the classes of {\em nilpotent and solvable} \footnote{We refer to these families of matroids as ``nilpotent” and ``solvable” by analogy with nilpotent and solvable groups, as both are characterized by the eventual triviality of a certain descending chain of submatroids. Furthermore, every nilpotent matroid is solvable. It is important to note, however, that these families of matroids are purely combinatorial.} matroids, respectively. 
 The families of nilpotent and solvable matroids generalize the notions of forest configurations from Definition~\ref{forest} and inductively connected line arrangements \cite{nazir2012connectivity}, to higher dimensions.

 We begin by defining solvable and nilpotent matroids, where the notions of $\Li_{M}$, $\Li_{p}$, and $a_{p}$ are introduced in Definition~\ref{general} and Notation~\ref{m6}.

\begin{definition}[Definitions~\ref{k2} and~\ref{m7}]
For a matroid $M$ on $[n]$, we define 
\[S_{M}=\{p \in [n]: \lvert \Li_p \rvert > 1\}, \quad \text{and} \quad Q_{M}=\{p \in [n]: a_{p}\leq 0\}.\] 
We then consider the following chains of submatroids of $M$: 
\begin{equation*}
\begin{aligned}
&M_{0}=M, \ \quad M_{1}=S_{M},\quad \text{and }\quad M_{j+1}=S_{M_{j}} \quad \ \text{ for all $j\geq 1$.}\\
& M^{0}=M,\quad M^{1}=Q_{M},\quad \text{and }\quad M^{j+1}=Q_{M^{j}} \quad \text{ for all $j\geq 1$.}
\end{aligned}
\end{equation*}
We say that $M$ is {\em nilpotent} if $M_{j}=\emptyset$ for some $j$, and 
{\em solvable} if $M^{j}=\emptyset$ for some $j$.
\end{definition}

We now provide some motivation for these definitions. As introduced in Definition~\ref{general}, $\Li_{M}$ denotes the set of all subspaces of $M$. Informally, if $p \in [n]$ lies in a subspace $l \in \Li_{M}$, then $l$ imposes a constraint on $p$, as any realization of $M$ must place $p$ in the linear subspace~$l$.


Nilpotent and solvable matroids are defined through iterative processes that remove relatively unconstrained points in realizations, continuing until none remain. The distinction between these two notions lies in the criterion for what qualifies as unconstrained:

\begin{itemize}
\item In a nilpotent matroid, unconstrained points are those contained in a unique element $l$ of $\Li_{M}$. Geometrically, such points are straightforward to realize, as they are subject to a single linear constraint; any generic point on the subspace realizing $l$ works.
\item In a solvable matroid, unconstrained points are those for which the sum of the codimensions of the elements of $\Li_{M}$ they belong to is strictly less than the rank. This condition ensures that the intersection of the associated linear subspaces has nonempty projectivization, guaranteeing a location where the point may be placed.
\end{itemize}

By construction, both nilpotent and solvable matroids have a natural filtration, which plays a key role in the inductive structure of our arguments.

\medskip
We study the realizability of these matroids and the irreducible decomposition of their realization spaces, establishing the following results. For reference, we briefly mention the notions of the realization space $\Gamma_{M}$, the matroid variety $V_{M}$, and the circuit variety $V_{\mathcal{C}(M)}$ associated with a matroid $M$, as introduced in Subsection~\ref{sec reali}. Additionally, we refer to the definitions of paving matroids, forest configurations, and the Grassmann-Cayley ideal $G_{M}$ from Definitions~\ref{pav},~\ref{forest}, and~\ref{gm}, respectively.

\begin{theoremA}
Let $M$ be a nilpotent matroid and let $\Gamma_{M},V_{M}$ and $V_{\mathcal{C}(M)}$ denote its realization space, matroid variety and circuit variety, respectively, as defined in Subsection~\ref{sec reali}. Then the following statements hold:
\begin{itemize}
\item $M$ is realizable, and $\Gamma_{M}$ is irreducible. \hfill
\textup{(Theorem~\ref{nil rea})}
\item If $M$ is a nilpotent paving matroid in which any three dependent hyperplanes have an empty intersection, then the following equality holds:
\[V_{M}=V_{\mathcal{C}(M)}\tag{Theorem~\ref{je}}.\]
\item If $M$ is a forest configuration, then $M$ is nilpotent and the following equality holds:
\[I(V_{M})=\sqrt{I(V_{\mathcal{C}(M)})+G_{M}}\tag{Corollary~\ref{incl} and Theorem~\ref{gc 3}},\]
where $G_{M}$, as defined in Definition~\ref{gm}, is the ideal generated by polynomials derived from the Grassmann-Cayley algebra.
\end{itemize}
\end{theoremA}


\begin{theoremB}
Let $M$ be 
matroid. The following statements hold:
\begin{itemize}
\item If  
$M$ is a solvable matroid of rank three, then $\Gamma_{M}$ is irreducible. \hfill \textup{(Theorem~\ref{teosol})}
\smallskip
\item  Let $M$ be a paving matroid in which any three dependent hyperplanes have an empty intersection. Then $M$ is solvable, and 
$\Gamma_{M}$ is irreducible. \hfill \textup{(Example~\ref{ktu 2} and Theorem~\ref{irreducible})}
\end{itemize}
\end{theoremB}

We regard the empty set as irreducible in the Zariski topology; therefore, our statement does not imply that all rank-three solvable matroids are realizable, which is certainly not the case. For instance, the non-Pappus matroid is a non-realizable solvable matroid of rank three; see Figure~1(b) in \textup{\cite{rosen2020algebraic}} and the discussion therein for further details.

\medskip

\noindent{\bf Related work.} We 
now discuss some related works. The families of nilpotent and solvable matroids generalize the notions of forest configurations and inductively connected line arrangements from \cite{clarke2021matroid} and \cite{nazir2012connectivity}, respectively.
Theorem~\ref{nil rea} establishes the realizability of nilpotent matroids,  thereby generalizing \textup{\cite[Proposition~5.7]{clarke2021matroid}}, which proved the realizability of forest 
configurations. 
Furthermore, \textup{\cite[Theorem~5.11]{clarke2021matroid}} established that for forest 
configurations without points of degree greater than two,
the matroid variety and the circuit variety coincide. In Theorem~\ref{je}, we 
extend this result to the family of nilpotent paving matroids without points of degree greater than two. 
Moreover, Theorem~\ref{gc 3} provides the defining equations of matroid varieties associated with forest 
configurations, even when points of degree greater than two appear. Moreover, the families of nilpotent and solvable matroids play a role in the study of hypergraph variety decompositions; see Section~\ref{applications}.



\medskip\noindent{\bf Outline.}
 Section~\ref{sec 2} introduces key concepts such as paving matroids 
 and matroid varieties. 
In Section~\ref{sec 4}, we introduce the nilpotnet matroids, establishing their realizability and the irreducibility of their matroid varieties. Moreover, we study the defining equations of nilpotent paving matroids without points of degree greater than two and forest 
configurations. 
In Section~\ref{solv}, we introduce solvable matroids and examine the irreducible decomposition of matroid varieties associated with a particular subfamily. In Section~\ref{applications}, we explore the applications of these matroid families to the study of hypergraph varieties.

\medskip\noindent{\bf Acknowledgement.}
We thank the referees for their very helpful suggestions and remarks. E.L. is supported by the FWO PhD fellowship 1126125N. F.M. was partially supported by the FWO grants G0F5921N (Odysseus) and G023721N, and the iBOF/23/064 grant from KU Leuven.

\section{Preliminaries}\label{sec 2}

\begin{notation}\label{notation 1}
We work over the complex field $\CC$. 
We denote the set $\{1, \ldots, n\}$ by $[n]$ and the collection of all its $r$-subsets by $\textstyle\binom{[n]}{r}$ for any $r\leq n$. For an $r \times n$ matrix $X = (x_{i,j})$ of indeterminates, $\CC[X]$ refers to the polynomial ring on the variables $x_{i,j}$. For a given pair of subsets $A \subseteq [r]$ and $B \subseteq [n]$ with $\lvert A \rvert = \lvert B \rvert$, $[A|B]_X$ denotes the minor of $X$ with rows indexed by $A$ and columns indexed by $B$. For convenience, we adopt the convention that the empty set is considered irreducible in the Zariski topology throughout this note.
\end{notation}

\subsection{Matroids}

We will outline some theoretical results related to matroids. While there are many equivalent definitions of matroids, each useful in different contexts, we will present only one here. For foundational concepts regarding matroids, we refer the reader to \cite{Oxley, piff1970vector}.

\begin{definition}\normalfont
A matroid can be characterized by its {\em circuits}. Specifically, a matroid $M$ consists of a ground set $[n]$ and a collection $\mathcal{C}$ of subsets of $[n]$, called
circuits, that satisfy the following axioms: 
\begin{itemize}
\item $\emptyset \not\in\mathcal{C}$.
\item If $C_1,C_2 \in \mathcal{C}$ and $C_1\subseteq C_2$, then $C_1=C_2$.
\item  
 If $e\in C_1\cap C_2$ with $C_{1}\neq C_{2}\in \mathcal{C}$, there exists $C_3\in\mathcal{C}$ such that $C_3\subseteq (C_1\cup C_2)-e$. 
\end{itemize}
\end{definition}

We will require the following notions. For a comprehensive source that has the necessary background about matroids, we refer the reader to \cite{Oxley}.

\begin{definition}\normalfont
Let $M$ be a matroid on the ground set $[n]$.
\begin{itemize}
\item A subset of $[n]$ that contains a circuit is called {\em dependent}, while any subset that does not contain a circuit is called {\em independent}. The collection of all dependent sets of $M$ is denoted by $\mathcal{D}(M)$. 
\item A {\em basis} is a maximal independent subset of $[n]$, and the set of all bases is denoted by $\mathcal{B}(M)$.
\item The collection of circuits of $M$ is denoted by $\mathcal{C}(M)$.
\item For any $F\subset [n]$, the {\em rank} of $F$, denoted $\rank(F)$, is defined as the size of the largest independent set contained in $F$. Observe that $\rank([n])$ is the size of any basis, which we denote as $\rank(M)$.
\item For a subset $F\subset [n]$, we define its {\em closure} $\closure{F}$ as the set $\{x\in [n]:\rank(F\cup \{x\})=\rank(F)\}$. A set is called a {\em flat} if it is equal to its own closure.
\item An element $x\in [n]$ satisfying $\rank(\{x\})=0$ is called a {\em loop}.
\item An element $x\in [n]$ that does not belong to any circuit of $M$ is called a {\em coloop}.
\item Two elements $x,y\in [n]$ are called {\em parallel} if $\rank(\{x,y\})=1$.
\end{itemize}
\end{definition}

Given a matroid $M$ on $[n]$, we define two fundamental operations: restriction and deletion.
\begin{definition}
\normalfont \label{subm}
For any subset $S \subseteq [n]$, the {\em restriction} of $M$ to $S$ is the matroid 
on $S$ whose rank function is the restriction of the rank function of $M$ to $S$. This matroid is referred to as the restriction of $M$ to $S$. 
Unless otherwise specified, we assume that subsets of $[d]$ possess this structure and refer to them as \textit{submatroids} of $M$. This submatroid is denoted by $M|S$, or simply $S$ when the context is clear. The {\em deletion} of $S$, is denoted by $M\setminus S$, which corresponds to $M|([n]\setminus S)$.
\end{definition}

\begin{definition}\normalfont\label{general}
Let $M$ be a matroid of rank $r$ on $[n]$. We refer to the elements in the ground set of 
$M$ as points. 
We consider the following equivalence relation on  circuits of $M$ of size less than $r+1$: 
$$C_{1}\sim C_{2} \Longleftrightarrow \closure{C_{1}}=\closure{C_{2}}.$$
We fix the following notation:
\begin{itemize}
\item We say that an equivalence class $l$ is a \textit{$k$-subspace} if $\rank(C) = k$ for any circuit $C \in l$. In this case, we denote the rank of $l$ by $\rank(l) = k$. We denote by $\Li_{M}$ the set of all \textit{subspaces} of $M$.

\item We say that $p \in [n]$ belongs to the \textit{subspace} $l$ if there exists a circuit $C \in l$ such that $p \in C$. For $p \in [n]$, let $\Li_{p}$ denote the set of all \textit{subspaces} of $M$ that contain the point $p$. We then define the \textit{degree} of $p$ as the size of $\Li_{p}$.

\item Let $\gamma$ be a collection of vectors of $\CC^{r}$ indexed by $[n]$. We denote by $\gamma_{p}$ the vector in $\gamma$ corresponding to the point $p$. For 
$l \in \Li_{M}$, we denote by $\gamma_{l}$ the subspace spanned by $\{\gamma_{p} \mid p \in l\}$.

 \item A vector $q \in \mathbb{C}^{r}$ is called \textit{outside} $\gamma$ if it does not belong to any subspace $\gamma_{l}$ for any $l \in \Li_{M}$.

\end{itemize}
For ease of notation, when it is clear, we drop the index $M$ and 
simply write $\Li$. 
\end{definition}

We fix the following notation that we will use throughout the paper.

\begin{notation}\label{initial notation}
Given a matroid $M$ of rank $r$ on the ground set $[n]$, we associate an $r\times n$ matrix $X$ of indeterminates, where the columns are indexed by the elements of $[n]$. For any subset $\{p_{1},\ldots,p_{k}\}\subset [n]$ with $k\leq r$, and vectors $v_{1},\ldots,v_{r-k}\in \CC^{r}$, 
we define the expression $[p_{1},\ldots,p_{k},v_{1},\ldots,v_{r-k}]\in \CC[X]$ as the determinant of the matrix formed by taking the columns of $X$ corresponding to $p_{1},\ldots,p_{k}$ 
and appending the vectors $v_{1},\ldots,v_{r-k}$ as additional columns. Note that $[p_{1},\ldots,p_{k},v_{1},\ldots,v_{r-k}]$ is a polynomial in $\CC[X]$, and not a scalar number.
\end{notation}

One particular family of matroids we will study in detail is the family of paving matroids~\cite{Oxley}.

\begin{definition}
\normalfont 
\label{pav}
A matroid $M$ of rank $r$ is called a \textit{paving matroid} if every circuit of $M$ has size either $r$ or $r+1$. In this case, we refer to it as an $r$-paving matroid. For paving matroids, the set of subspaces $\Li_{M}$, as described in Definition~\ref{general}, coincides with the set of {\em dependent hyperplanes} of $M$, which are maximal subsets of points of size at least $r$ in which every subset of $r$ elements forms a circuit. These subspaces are called dependent hyperplanes.

For $r=3$, a paving matroid is simply a rank-three simple matroid. Throughout this note, we assume all rank-three matroids are simple, referring to their ground set elements as \emph{points} and their dependent hyperplanes as \emph{lines}.
\end{definition}

\begin{example}
\label{three lines}
The configuration depicted in Figure~\ref{fig:combined} (Left) represents a rank-three matroid $M$ with points $[7]$ and lines $\Li_{M} = \{\{7,1,2\}, \{7,3,4\}, \{7,5,6\}\}$. 
\end{example}

\subsection{Realization spaces of matroids and matroid  varieties}
\label{sec reali}

We begin by recalling the realization spaces of matroids, followed by circuit and matroid varieties.

\begin{definition}
Let $M$ be a matroid of rank $r$ on the ground set $[n]$.
A realization of $M$ 
is a collection of vectors $\gamma=\{\gamma_{1},\ldots,\gamma_{n}\}\subset \CC^{r}$ such that
\[\{\gamma_{i_{1}},\ldots,\gamma_{i_{p}}\}\ \text{is linearly dependent} \Longleftrightarrow \{i_{1},\ldots,i_{p}\} \ \text{is a dependent set of $M$.}\]
The realization space of $M$ 
is defined as $\Gamma_{M}=\{\gamma \subset \CC^{r}: \gamma \ \text{is a realization of $M$}\}$.
Each element of $\Gamma_{M}$ can be represented as an $r\times n$ matrix over $\CC$. The matroid variety $V_{M}$ of $M$ is defined as the Zariski closure in $\CC^{rn}$ of $\Gamma_{M}$. The  associated ideal $I_{M}=I\pare{V_{M}}$ is called the matroid ideal. 
\end{definition}


\begin{definition}\normalfont\label{cir}
Let $M$ be a matroid of rank $r$ on $[n]$. 
Consider the $r\times n$ matrix $X=\pare{x_{i,j}}$ of indeterminates. We define the \textit{circuit ideal} of $M$ as:
$$ I_{\mathcal{C}(M)} = \{ [A|B]_X:\ B \in \mathcal{C}(M),\ A \subseteq [r],\ \text{and}\ |A| = |B| \}. $$
Note that the polynomials in $I_{\mathcal{C}(M)}$ vanish in any realization of $M$, which implies that $I_{\mathcal{C}(M)} \subseteq I_{M}$. 
We say that $\gamma$, a collection of vectors of $\CC^{r}$ indexed by $[n]$, includes the dependencies of $M$ if it satisfies:
\[\set{i_{1},\ldots,i_{k}}\  \text{is a dependent set of $M$} \Longrightarrow \set{\gamma_{i_{1}},\ldots,\gamma_{i_{k}}}\ \text{is linearly dependent}. \] 
The {\em circuit variety} is $V_{\mathcal{C}(M)}=V(I_{\mathcal{C}(M)})=\{\gamma:\text{$\gamma$ includes the dependencies of $M$}\}$.
\end{definition}

We conclude this subsection by highlighting the relation between the study of the realization spaces of matroids with those of hyperplane arrangements.

\begin{remark}\label{arrangemnet}
There exists a natural correspondence between the realization spaces of matroids and those of hyperplane arrangements. Each vector collection realizing a matroid $M$ corresponds to a hyperplane arrangement in the dual vector space; see e.g.~\cite{bokowski1989computational,richter1999universality,Emiliano-Fatemeh}. 
Consequently, for any realizable matroid $M$, there exists a correspondence between the realization space of $M$ and the realization space of any hyperplane arrangement having $M$ as its associated matroid. Under this correspondence, the realization spaces of line arrangements correspond with those of rank-three matroids.
\end{remark}

\subsection{Grassmann-Cayley algebra}\normalfont

We now review the notion of Grassmann-Cayley algebra from \textup{\cite{computationalgorithms}}. The Grassmann-Cayley algebra is the exterior algebra $\bigwedge (\mathbb{C}^{d})$, equipped with two operations: the \textit{join} and the \textit{meet}. The join, or extensor, of vectors $v_{1}, \ldots, v_{k}$ is denoted by $v_{1}  \cdots v_{k}$. The meet operation, denoted by $\wedge$, is defined for two extensors $v = v_{1} \cdots v_{k}$ and $w = w_{1} \cdots w_{j}$, with lengths $k$ and $j$ respectively, where $j + k \geq d$, as:
$$v\wedge w=\sum_{\sigma\in \mathcal{S}(k,j,d)}
\corch{v_{\sigma(1)}\cdots v_{\sigma(d-j)}w_{1}\cdots w_{j}}\cdot v_{\sigma(d-j+1)}\cdots v_{\sigma(k)}$$
where $ \mathcal{S}(k,j,d)$ denotes the set of all permutations of $\corch{k}$ that satisfy $\sigma(1)<\cdots <\sigma(d-j)$ and $\sigma(d-j+1)<\cdots <\sigma(d)$. When $j + k < d$, the meet is defined as $0$. There is a correspondence between the extensor $v = v_{1} \cdots v_{k}$ and the subspace $\overline{v} = \langle v_{1}, \ldots, v_{k} \rangle$ that they generate. This correspondence satisfies the following properties:

\begin{lemma}\label{klj}

Let $v=v_{1}\cdots v_{k}$ and $w=w_{1}\cdots w_{j}$ be two extensors with $j+k\geq d$. Then we have:
\begin{itemize}
\item The extensor $v$ is equal to zero if and only if the vectors $v_{1},\ldots,v_{k}$ are linearly dependent.
\item Any extensor $v$ is uniquely determined by $\overline{v}$ up to a scalar multiple.
\item The meet of two extensors is again an extensor.
\item We have that $v\wedge w\neq 0$ if and only if $\ip{\overline{v}}{\overline{w}}=\CC^{d}$. In this case, we have $\overline{v}\cap\overline{w}=\overline{v\wedge w}.$
\end{itemize}
\end{lemma}

\begin{example}\label{gc3}
Consider the rank-three matroid depicted on the left of Figure~\ref{fig:combined}. Since, in any realization, the lines $\{12, 34, 56\}$ are concurrent, the point of intersection of $12$ and $34$ lies on $56$. Thus, the concurrency of the lines $\{12, 34, 56\}$ is equivalent to the vanishing of the following polynomial:
$$(34)\wedge (12)\vee 56=(\corch{312}4-\corch{412}3)\vee 56=\corch{123}\corch{456}-\corch{124}\corch{356}.$$
Therefore, we know that this polynomial is contained in the ideal of the matroid variety.
\end{example}

By repeatedly applying the last item of Lemma~\ref{klj}, we obtain the following result, 
demonstrating that for given subspaces whose dimensions meet the lemma's conditions, their Grassmann-Cayley algebra expression is non-vanishing if and only if their intersection has the \textit{expected dimension}.
\begin{lemma}[Expected dimension]\label{kñl}
Let $v_{1},\ldots,v_{k}$ be extensors associated to subspaces $V_{1},\ldots,V_{k}$ of $\CC^{d}$ of dimensions $a_{1},\ldots,a_{k}$ respectively. Suppose that $\textstyle \sum_{i=1}^{k}a_{i}-d(k-1)\geq 0$. Then, we have that:
\[ \dim(\bigcap_{i=1}^{k}V_{i})=\sum_{i=1}^{k}a_{i}-d(k-1)\quad\text{if and only if}\quad \bigwedge_{i=1}^{k}v_{i}\neq 0.\] 
\end{lemma}

\section{Nilpotent matroids}\label{sec 4}

In this section, in Definition~\ref{k2}, we introduce a new family of matroids called \textit{nilpotent}. We prove, in Theorem~\ref{nil rea}, that the matroid varieties of nilpotent matroids are realizable and irreducible. Additionally, we study the defining equations of their matroid varieties.
Moreover, in Theorem~\ref{gc 3}, we provide a complete set of defining equations for forest 
configurations.

\subsection{Definition and basic properties} 
We now introduce the family of \textit{nilpotent matroids} and examine their basic properties. Recall the invariants from Definition~\ref{general}, particularly 
\(\Li\), representing the sets of subspaces.

\begin{notation}\label{k1}
Let $M$ be a matroid on $[n]$. We denote the set of points $\{p \in [n] \mid \lvert \Li_p \rvert > 1\}$ by $S_M$. For each $l \in \Li$, we denote $S_M \cap l$ by $S_l$.
(Note that, $S_{N}\subset S_{M}$ for every submatroid $N\subset M$.)
\end{notation}

\begin{definition}[Nilpotent matroid]\label{k2}
Let $M$ be a matroid. We define the \textit{nilpotent chain} of $M$ as the following chain of submatroids of $M$: 
$$M_{0}=M,\quad M_{1}=S_{M},\quad \text{and }\ M_{j+1}=S_{M_{j}} \ \text{ for every $j\geq 1$.}$$
We call $M$ \textit{nilpotent} if $M_{j}=\emptyset$ for some $j$. In this case, we denote the length of the chain by $l_{n}(M)$.
\end{definition}

In the preceding definition, we identify each subset of points $M_{j}$ with the submatroid it defines.

\begin{example}
The two 
rank-three matroids on the right in Figure~\ref{fig:combined} are not nilpotent because all their points have a degree greater than one. In contrast, the matroid on the left is nilpotent, with a nilpotent chain of length two.
\end{example}

\begin{figure}[h]
    \centering
    \begin{subfigure}[b]{0.3\textwidth}
        \centering
        \begin{tikzpicture}[x=0.75pt,y=0.75pt,yscale=-1,xscale=1]

\tikzset{every picture/.style={line width=0.75pt}} 

\draw    (81.69,116.61) -- (191.16,174.3) ;
\draw    (77,131.88) -- (224,131.88) ;
\draw    (80.13,150.55) -- (191.16,79.27) ;
\draw [fill={rgb, 300:red, 0; green, 0; blue, 180}, fill opacity=1]
(107.34,131.2) .. controls (107.34,133.07) and (108.63,134.58) .. (110.23,134.58) .. controls (111.83,134.58) and (113.12,133.07) .. (113.12,131.2) .. controls (113.12,129.33) and (111.83,127.81) .. (110.23,127.81) .. controls (108.63,127.81) and (107.34,129.33) .. (107.34,131.2) -- cycle ;
\draw [fill={rgb, 300:red, 0; green, 0; blue, 180}, fill opacity=1]
(142.62,149.93) .. controls (142.62,151.8) and (143.91,153.32) .. (145.51,153.32) .. controls (147.11,153.32) and (148.4,151.8) .. (148.4,149.93) .. controls (148.4,148.06) and (147.11,146.55) .. (145.51,146.55) .. controls (143.91,146.55) and (142.62,148.06) .. (142.62,149.93) -- cycle ;
\draw [fill={rgb, 300:red, 0; green, 0; blue, 180}, fill opacity=1]
(173.7,166.79) .. controls (173.7,168.66) and (174.99,170.18) .. (176.59,170.18) .. controls (178.19,170.18) and (179.48,168.66) .. (179.48,166.79) .. controls (179.48,164.92) and (178.19,163.41) .. (176.59,163.41) .. controls (174.99,163.41) and (173.7,164.92) .. (173.7,166.79) -- cycle ;
\draw [fill={rgb, 300:red, 0; green, 0; blue, 180}, fill opacity=1] (201.42,131.2) .. controls (201.42,133.07) and (202.71,134.58) .. (204.31,134.58) .. controls (205.91,134.58) and (207.2,133.07) .. (207.2,131.2) .. controls (207.2,129.33) and (205.91,127.81) .. (204.31,127.81) .. controls (202.71,127.81) and (201.42,129.33) .. (201.42,131.2) -- cycle ;
\draw [fill={rgb, 300:red, 0; green, 0; blue, 180}, fill opacity=1] (167.82,131.2) .. controls (167.82,133.07) and (169.11,134.58) .. (170.71,134.58) .. controls (172.31,134.58) and (173.6,133.07) .. (173.6,131.2) .. controls (173.6,129.33) and (172.31,127.81) .. (170.71,127.81) .. controls (169.11,127.81) and (167.82,129.33) .. (167.82,131.2) -- cycle ;
\draw [fill={rgb, 300:red, 0; green, 0; blue, 180}, fill opacity=1] (177.06,87.18) .. controls (177.06,89.05) and (178.35,90.56) .. (179.95,90.56) .. controls (181.55,90.56) and (182.84,89.05) .. (182.84,87.18) .. controls (182.84,85.31) and (181.55,83.79) .. (179.95,83.79) .. controls (178.35,83.79) and (177.06,85.31) .. (177.06,87.18) -- cycle ;
\draw [fill={rgb, 300:red, 0; green, 0; blue, 180}, fill opacity=1] (145.14,105.91) .. controls (145.14,107.78) and (146.43,109.3) .. (148.03,109.3) .. controls (149.63,109.3) and (150.92,107.78) .. (150.92,105.91) .. controls (150.92,104.04) and (149.63,102.52) .. (148.03,102.52) .. controls (146.43,102.52) and (145.14,104.04) .. (145.14,105.91) -- cycle ;

\draw (166.25,71.57) node [anchor=north west][inner sep=0.75pt]   [align=left] {{\scriptsize 1}};
\draw (206.28,115.51) node [anchor=north west][inner sep=0.75pt]   [align=left] {{\scriptsize 3}};
\draw (138.91,153.35) node [anchor=north west][inner sep=0.75pt]   [align=left] {{\scriptsize 6}};
\draw (170.08,173.62) node [anchor=north west][inner sep=0.75pt]   [align=left] {{\scriptsize 5}};
\draw (173.44,116.45) node [anchor=north west][inner sep=0.75pt]   [align=left] {{\scriptsize 4}};
\draw (106.8,111.36) node [anchor=north west][inner sep=0.75pt]   [align=left] {{\scriptsize 7}};
\draw (136.66,90.59) node [anchor=north west][inner sep=0.75pt]   [align=left] {{\scriptsize 2}};

   \end{tikzpicture}
        \label{fig:quadrilateral}
    \end{subfigure}
    \hfill
    \begin{subfigure}[b]{0.3\textwidth}
        \centering

\tikzset{every picture/.style={line width=0.75pt}} 

\begin{tikzpicture}[x=0.75pt,y=0.75pt,yscale=-1,xscale=1]

\draw [line width=0.75]    (247.65,123.92+60) -- (201.96,207.53+60) ;
\draw [line width=0.75]    (247.65,123.92+60) -- (291.68,207.53+60) ;
\draw [line width=0.75]    (219.23,174.98+60) -- (291.68,207.53+60) ;
\draw [line width=0.75]    (274.41,174.98+60) -- (201.96,207.53+60) ;
\draw [fill={rgb, 300:red, 0; green, 0; blue, 180}, fill opacity=1] (244.87,123.92+60) .. controls (244.87,125.68+60) and (246.12,127.11+60) .. (247.65,127.11+60) .. controls (249.19,127.11+60) and (250.44,125.68+60) .. (250.44,123.92+60) .. controls (250.44,122.15+60) and (249.19,120.73+60) .. (247.65,120.73+60) .. controls (246.12,120.73+60) and (244.87,122.15+60) .. (244.87,123.92+60) -- cycle ;
\draw [fill={rgb, 300:red, 0; green, 0; blue, 180}, fill opacity=1] (244.31,187.1+60) .. controls (244.31,188.87+60) and (245.56,190.29+60) .. (247.1,190.29+60) .. controls (248.64,190.29+60) and (249.88,188.87+60) .. (249.88,187.1+60) .. controls (249.88,185.34+60) and (248.64,183.91+60) .. (247.1,183.91+60) .. controls (245.56,183.91+60) and (244.31,185.34+60) .. (244.31,187.1+60) -- cycle ;
\draw [fill={rgb, 300:red, 0; green, 0; blue, 180}, fill opacity=1] (216.45,174.98+60) .. controls (216.45,176.74+60) and (217.69,178.17+60) .. (219.23,178.17+60) .. controls (220.77,178.17+60) and (222.02,176.74+60) .. (222.02,174.98+60) .. controls (222.02,173.21+60) and (220.77,171.79+60) .. (219.23,171.79+60) .. controls (217.69,171.79+60) and (216.45,173.21+60) .. (216.45,174.98+60) -- cycle ;
\draw [fill={rgb, 300:red, 0; green, 0; blue, 180}, fill opacity=1] (271.62,174.98+60) .. controls (271.62,176.74+60) and (272.87,178.17+60) .. (274.41,178.17+60) .. controls (275.94,178.17+60) and (277.19,176.74+60) .. (277.19,174.98+60) .. controls (277.19,173.21+60) and (275.94,171.79+60) .. (274.41,171.79+60) .. controls (272.87,171.79+60) and (271.62,173.21+60) .. (271.62,174.98+60) -- cycle ;
\draw [fill={rgb, 300:red, 0; green, 0; blue, 180}, fill opacity=1] (288.9,207.53+60) .. controls (288.9,209.29+60) and (290.14,210.72+60) .. (291.68,210.72+60) .. controls (293.22,210.72+60) and (294.47,209.29+60) .. (294.47,207.53+60) .. controls (294.47,205.77+60) and (293.22,204.34+60) .. (291.68,204.34+60) .. controls (290.14,204.34+60) and (288.9,205.77+60) .. (288.9,207.53+60) -- cycle ;
\draw [fill={rgb, 300:red, 0; green, 0; blue, 180}, fill opacity=1] (199.17,207.53+60) .. controls (199.17,209.29+60) and (200.42,210.72+60) .. (201.96,210.72+60) .. controls (203.49,210.72+60) and (204.74,209.29+60) .. (204.74,207.53+60) .. controls (204.74,205.77+60) and (203.49,204.34+60) .. (201.96,204.34+60) .. controls (200.42,204.34+60) and (199.17,205.77+60) .. (199.17,207.53+60) -- cycle ;

\draw (243.29,107.07+60) node [anchor=north west][inner sep=0.75pt]   [align=left] {{\scriptsize 1}};
\draw (208.77,165.67+60) node [anchor=north west][inner sep=0.75pt]   [align=left] {{\scriptsize 2}};
\draw (192.94,205.77+60) node [anchor=north west][inner sep=0.75pt]   [align=left] {{\scriptsize 3}};
\draw (241.88,169.35+60) node [anchor=north west][inner sep=0.75pt]   [align=left] {{\scriptsize 4}};
\draw (277.98,164.77+60) node [anchor=north west][inner sep=0.75pt]   [align=left] {{\scriptsize 5}};
\draw (292.84,206.21+60) node [anchor=north west][inner sep=0.75pt]   [align=left] {{\scriptsize 6}};

\end{tikzpicture}
        \label{fig:pascal}
    \end{subfigure}
    \hfill
    \begin{subfigure}[b]{0.3\textwidth}
        \centering

\tikzset{every picture/.style={line width=0.75pt}} 

\begin{tikzpicture}[x=0.75pt,y=0.75pt,yscale=-1,xscale=1]

\draw  [dash pattern={on 4.5pt off 4.5pt}] (171,191.4) .. controls (171,167.83) and (209.5,148.72) .. (257,148.72) .. controls (304.5,148.72) and (343,167.83) .. (343,191.4) .. controls (343,214.98) and (304.5,234.09) .. (257,234.09) .. controls (209.5,234.09) and (171,214.98) .. (171,191.4) -- cycle ;
\draw    (193.78,161.68) -- (247.97,234.09) ;
\draw    (190.63,182.26) -- (313.94,208.17) ;
\draw    (302.95,154.82) -- (247.97,234.09) ;
\draw    (250.32,148.72) -- (185.92,215.03) ;
\draw    (193.78,161.68) -- (285.67,232.56) ;
\draw    (302.95,154.82) -- (185.92,215.03) ;
\draw    (250.32,148.72) -- (285.67,232.56) ;
\draw [fill={rgb, 300:red, 0; green, 0; blue, 180}, fill opacity=1] (191.42,161.68) .. controls (191.42,162.94) and (192.47,163.96) .. (193.78,163.96) .. controls (195.08,163.96) and (196.13,162.94) .. (196.13,161.68) .. controls (196.13,160.41) and (195.08,159.39) .. (193.78,159.39) .. controls (192.47,159.39) and (191.42,160.41) .. (191.42,161.68) -- cycle ;
\draw [fill={rgb, 300:red, 0; green, 0; blue, 180}, fill opacity=1] (229.9,191.4) .. controls (229.9,192.67) and (230.96,193.69) .. (232.26,193.69) .. controls (233.56,193.69) and (234.62,192.67) .. (234.62,191.4) .. controls (234.62,190.14) and (233.56,189.12) .. (232.26,189.12) .. controls (230.96,189.12) and (229.9,190.14) .. (229.9,191.4) -- cycle ;
\draw [fill={rgb, 300:red, 0; green, 0; blue, 180}, fill opacity=1] (283.31,232.56) .. controls (283.31,233.82) and (284.37,234.85) .. (285.67,234.85) .. controls (286.97,234.85) and (288.02,233.82) .. (288.02,232.56) .. controls (288.02,231.3) and (286.97,230.27) .. (285.67,230.27) .. controls (284.37,230.27) and (283.31,231.3) .. (283.31,232.56) -- cycle ;
\draw [fill={rgb, 300:red, 0; green, 0; blue, 180}, fill opacity=1] (245.61,234.09) .. controls (245.61,235.35) and (246.67,236.37) .. (247.97,236.37) .. controls (249.27,236.37) and (250.32,235.35) .. (250.32,234.09) .. controls (250.32,232.82) and (249.27,231.8) .. (247.97,231.8) .. controls (246.67,231.8) and (245.61,232.82) .. (245.61,234.09) -- cycle ;
\draw [fill={rgb, 300:red, 0; green, 0; blue, 180}, fill opacity=1] (210.27,186.83) .. controls (210.27,188.09) and (211.32,189.12) .. (212.63,189.12) .. controls (213.93,189.12) and (214.98,188.09) .. (214.98,186.83) .. controls (214.98,185.57) and (213.93,184.54) .. (212.63,184.54) .. controls (211.32,184.54) and (210.27,185.57) .. (210.27,186.83) -- cycle ;
\draw [fill={rgb, 300:red, 0; green, 0; blue, 180}, fill opacity=1] (183.57,215.03) .. controls (183.57,216.29) and (184.62,217.32) .. (185.92,217.32) .. controls (187.22,217.32) and (188.28,216.29) .. (188.28,215.03) .. controls (188.28,213.77) and (187.22,212.74) .. (185.92,212.74) .. controls (184.62,212.74) and (183.57,213.77) .. (183.57,215.03) -- cycle ;
\draw [fill={rgb, 300:red, 0; green, 0; blue, 180}, fill opacity=1] (300.59,154.82) .. controls (300.59,156.08) and (301.64,157.1) .. (302.95,157.1) .. controls (304.25,157.1) and (305.3,156.08) .. (305.3,154.82) .. controls (305.3,153.55) and (304.25,152.53) .. (302.95,152.53) .. controls (301.64,152.53) and (300.59,153.55) .. (300.59,154.82) -- cycle ;
\draw [fill={rgb, 300:red, 0; green, 0; blue, 180}, fill opacity=1] (247.97,148.72) .. controls (247.97,149.98) and (249.02,151.01) .. (250.32,151.01) .. controls (251.63,151.01) and (252.68,149.98) .. (252.68,148.72) .. controls (252.68,147.46) and (251.63,146.43) .. (250.32,146.43) .. controls (249.02,146.43) and (247.97,147.46) .. (247.97,148.72) -- cycle ;
\draw [fill={rgb, 300:red, 0; green, 0; blue, 180}, fill opacity=1] (269.17,199.02) .. controls (269.17,200.29) and (270.23,201.31) .. (271.53,201.31) .. controls (272.83,201.31) and (273.89,200.29) .. (273.89,199.02) .. controls (273.89,197.76) and (272.83,196.74) .. (271.53,196.74) .. controls (270.23,196.74) and (269.17,197.76) .. (269.17,199.02) -- cycle ;

\draw (183.96,150.51) node [anchor=north west][inner sep=0.75pt]   [align=left] {{\scriptsize 1}};
\draw (277.42,188.62) node [anchor=north west][inner sep=0.75pt]   [align=left] {{\scriptsize 8}};
\draw (226.37,193.19) node [anchor=north west][inner sep=0.75pt]   [align=left] {{\scriptsize 9}};
\draw (198.88,185.57) node [anchor=north west][inner sep=0.75pt]   [align=left] {{\scriptsize 7}};
\draw (291.55,228.25) node [anchor=north west][inner sep=0.75pt]   [align=left] {{\scriptsize 6}};
\draw (250.71,134.5) node [anchor=north west][inner sep=0.75pt]   [align=left] {{\scriptsize 5}};
\draw (176.1,210.72) node [anchor=north west][inner sep=0.75pt]   [align=left] {{\scriptsize 4}};
\draw (305.69,142.89) node [anchor=north west][inner sep=0.75pt]   [align=left] {{\scriptsize 3}};
\draw (248.36,234.35) node [anchor=north west][inner sep=0.75pt]   [align=left] {{\scriptsize 2}};

\end{tikzpicture}
        \label{fig:pencil}
    \end{subfigure}
    \caption{(Left) Three concurrent lines; (Center) Quadrilateral set; (Right) Pascal configuration.}
    \label{fig:combined}
\end{figure}

\vspace{-5mm}
\begin{lemma}\label{equiv 2}
Given a matroid $M$, the following statements hold:
\begin{itemize}
    \item[{\rm (i)}] $M$ is not nilpotent if and only if there exists a submatroid $\emptyset \subsetneq N \subset M$ with $S_{N}=N$.
    \item[{\rm (ii)}] Assume that $M$ is an $r$-paving matroid. Then $M$ is nilpotent if and only if every submatroid $N$ of full rank, with at least one dependent hyperplane, has a dependent hyperplane $l$ with $\size{S_{l}}\leq r-1$.
\end{itemize} 

\end{lemma}

\begin{proof}
(i) Suppose that there exists such a submatroid $N$. We will prove by induction that $N\subset M_{j}$ for every $j$. The base case is clear as $N\subset M=M_{0}$. For the inductive step suppose that $N\subset M_{j}$. Then $N=S_{N}\subset S_{M_{j}}=M_{j+1}$. This proves that $M$ is not nilpotent. Now suppose that $M$ is not nilpotent. Then the sequence 
\[M_0 \supset M_1 \supset M_2 \supset \cdots\]
stabilizes. Let $M_k$ be the subset at which this sequence stabilizes. Then $S_{M_k} = M_k$. Taking $N = M_k$, the claim follows.

(ii) Suppose that $M$ is not nilpotent. By (i), there exists a submatroid $N \neq \emptyset$ such that $N = S_N$, which implies that it has full rank. For this submatroid, there is no dependent hyperplane as described in the statement, because $\lvert S_l \rvert = \lvert l \rvert \geq r$ for every dependent hyperplane $l$ of $N$.

Now, suppose there exists a submatroid $N$ of full rank in which there is no such dependent hyperplane. 
Then, $\lvert S_l \rvert \geq r$ for every dependent hyperplane $l$ of $N$. Therefore, $N$ and $S_N$ have the same dependent hyperplanes. Consequently, we have $S_{S_N} = S_N$, hence $M$ is not nilpotent by (i).
\end{proof}

\begin{notation}\label{kb}
For an $r$-paving matroid $M$ on $[n]$ and a dependent hyperplane $l \in \Li$ with $\lvert S_l \rvert \leq r-1$, we use the notation $M - l$, with slight abuse of notation, to denote the submatroid $M\setminus (l\setminus S_{l})$. We refer to $M-l$ as the submatroid obtained by deleting the dependent hyperplane $l$. 
Note that $l$ is not a dependent hyperplane of $M - l$. 

For dependent hyperplanes $l_1, \ldots, l_k$, we denote by $M - \{l_1, \ldots, l_k\}$ the submatroid obtained by deleting these dependent hyperplanes in sequence. Here, $l_j$ is a dependent hyperplane of $M - \{l_1, \ldots, l_{j-1}\}$ for each $j \in [k]$.
\end{notation}

\begin{corollary}\label{rel}
Let $M$ be an $r$-paving matroid on $[n]$. Then $M$ is nilpotent if and only if there is a sequence of dependent hyperplanes $l_{1},\ldots,l_{k}$ such that:
\begin{itemize}
    \item $M-\set{l_{1},\ldots,l_{k}}$ does not contain any dependent hyperplane, and
    \item $\lvert{l_{j}\cap S_{M-\set{l_{1},\ldots,l_{j-1}}}\rvert} \leq r-1$ for every $j\in [k]$.
\end{itemize}

\end{corollary}

\begin{proof}
Suppose that $M$ is nilpotent. By Lemma~\ref{equiv 2}(ii), there exists a dependent hyperplane $l_1$ with $\lvert S_{l_1} \rvert \leq r-1$. If the submatroid $M - l_1$ contains at least two dependent hyperplanes, then it has full rank. Applying the same lemma, there exists a dependent hyperplane $l_2$ of $M - l_1$ with $\lvert S_{M - l_1} \cap l_2 \rvert \leq r-1$. Repeating the same process, we can find a sequence of dependent hyperplanes $l_1, \ldots, l_k$ such that $M - \{l_1, \ldots, l_{k-1}\}$ has at most one dependent hyperplane, and for each $j \in [k-1]$, we have
\[\lvert l_j \cap S_{M -\{l_1, \ldots, l_{j-1}\}} \rvert \leq r-1.\]
Since $M - \{l_1, \ldots, l_{k-1}\}$ has at most one dependent hyperplane, we have $S_{M -\{l_1, \ldots, l_{k-1}\}} = \emptyset$. If this submatroid has no dependent hyperplanes, then the sequence $l_1, \ldots, l_{k-1}$ satisfies the condition. If it does have a dependent hyperplane $l_k$, then the sequence $l_1, \ldots, l_k$ works as well.

Assume that a sequence of dependent hyperplanes $l_1, \ldots, l_k$ exists, and we need to prove that $M$ is nilpotent. We will use induction on the number of dependent hyperplanes.
As the base case, note that if $M$ has no dependent hyperplanes, it is trivially nilpotent. Now, assume the statement is true for any $r$-paving matroid with at most $m-1$ dependent hyperplanes. We need to prove it for an $r$-paving matroid $M$ with $m$ dependent hyperplanes.
By the assumption, the matroid $M - l_1$ has at most $m-1$ dependent hyperplanes. By the inductive hypothesis, $M - l_1$ is nilpotent. Since $S_M \subset ([n]  \setminus l_1)\cup S_{l_{1}}$, it follows that $S_M$ is nilpotent, implying that $M$ is nilpotent, as desired.
\end{proof}
 

We now recall the definition of a \textit{forest} 
configuration from \textup{\cite[Definition~5.1]{clarke2021matroid}}.

\begin{definition}\normalfont \label{forest}
A rank three matroid $M$ on $[n]$ is a \textit{forest} if there do not exist points $x_{1}, \ldots, x_{k} \in [n]$ and lines $l_{1}, \ldots, l_{k-1} \in \Li$ such that:
\begin{itemize}
\item $x_{1}=x_{k}$ and the points in $\{x_{1},\ldots,x_{k-1}\}$ are distinct.
\item $l_{j}\neq l_{j+1}$ for every $j\in [k-2]$ and $x_{j},x_{j+1}\in l_{j}$ for every $j\in [k-1]$. 
\end{itemize}
In other words, $M$ is a forest if there are no cycles.
\end{definition}

\begin{proposition}\label{fore}
A 
rank-three matroid $M$ is a forest if and only if there exists a sequence of lines $l_1, \ldots, l_k$ such that:
\begin{itemize}
    \item $M - \{l_1, \ldots, l_k\}$ contains no lines, and
    \item $\lvert l_j \cap S_{M - \{l_1, \ldots, l_{j-1}\}} \rvert \leq 1$ for every $j \in [k]$.
\end{itemize} 
\end{proposition}
\begin{proof}
If $M$ is a forest 
configuration, then by \textup{\cite[Lemma~5.3]{clarke2021matroid}}, there exists a line $l \in \Li$ with $\lvert S_l \rvert \leq 1$. Consequently, for any forest, there exists a sequence of lines 
with the desired properties. 

We will now prove the converse, by induction on the number of lines. Suppose the statement holds for any 
rank-three matroid with at most $m-1$ lines. We will show that it also holds for a 
rank-three matroid $M$ with $m$ lines.
Removing a line $l_1$ from $M$ results in a configuration $M - l_1$ that has at most $m-1$ lines. By the inductive hypothesis, $M - l_1$ is a forest.
Now, assume by contradiction that $M$ is not a forest. Then, there exists a cycle $C$ in $M$ as described in Definition~\ref{forest}. Since all points of $C$ are in $S_M$, $C$ is also a cycle in $M - l_1$, which contradicts the fact that $M - l_1$ is a forest.
This concludes the inductive step.
\end{proof}

By Corollary~\ref{rel} and Proposition~\ref{fore} we have that:
\begin{corollary}
\label{incl}
Every forest 
configuration is nilpotent.
\end{corollary}

\begin{example}
The two 
rank-three matroids depicted on the left of Figure~\ref{new figure} are examples of forest configurations. Consequently, by Corollary~\ref{incl}, they are also nilpotent.
\end{example}

\subsection{Realizability and irreducibility of nilpotent realization spaces}
Our goal in this subsection is to prove that nilpotent matroids are realizable and that their realization spaces are irreducible; see Theorem~\ref{nil rea}.
We now recall the notion of freely adding an element to a flat, also known as a principal extension; see \textup{\cite[Section~7.2]{Oxley}}.
\begin{definition} 
Let $F$ be a flat of a matroid $M$. We say that $M\rq$
is the matroid obtained from $M$ by freely adding an element $e$ to the flat $F$ if the flats of $M\rq$ are the following:
\begin{itemize}
\item flats $G$ of $M$ such that $F\not\subset G$,
\item sets $G\cup e$ where $G$ is a flat such that $F\subset G$, and
\item sets $G\cup e$ where $G$ is a flat such that $F\not\subset G$, and there is no flat $G\rq$  
of rank $\rank(G)+1$ with $G\cup F \subset G\rq$.
\end{itemize}
\end{definition}

\begin{lemma}\label{free}
Let $M$ be a matroid on $[n]$. Then the following hold:
\begin{itemize}
    \item[{\rm (i)}] Let $M\rq$
be the matroid obtained by freely adding an element to a flat
$F$ of $M$. Then $M\rq$
is realizable if and only if $M$ is realizable.
\item[{\rm (ii)}]
Let $M\rq$
be the matroid obtained by adding a coloop to $M$. If $M$ is realizable, then $M\rq$ is realizable.
\item[{\rm (iii)}] Let $p\notin S_{M}$. If 
$M\setminus p$ is realizable, then $M$ is realizable. 
\end{itemize}
\end{lemma}
\begin{proof} 
(i) and (ii) are well-known and follow from a result in \cite{piff1970vector}; see also \textup{\cite[Proposition~11.2.16]{Oxley}}.

(iii) 
We have the following cases.
If $\size{\Li_{p}}=0$ and $\rank(M)=\rank(M\setminus p)$, then $M$ is obtained from $M\setminus p$ by freely adding a point to the unique flat of rank $\rank(M)$. Therefore, $M$ is realizable by (i).
Similarly, if $\size{\Li_{p}}=\{l\}$ then $M$ is obtained from $M\setminus p$ by freely adding a point to the flat $\closure{l}$, making $M$ realizable.
If $\size{\Li_{p}}=0$ and $\rank(M)>\rank(M\setminus p)$ then $M$ is obtained from $M\setminus p$ by adding a coloop, hence $M$ is realizable by (ii).
\end{proof}

\begin{theorem}\label{nil rea}
Let $M$ be a nilpotent matroid. Then $M$ is realizable and $\Gamma_{M}$  
is irreducible. 
\end{theorem}
\begin{proof}
To prove that $M$ is realizable, we use induction on $[n]$, the size of the ground set.
The base case is 
$n=1$, where the result holds. Now suppose that it is true for nilpotent matroid $M$ on $[i]$ with $i\leq n-1$ and let $M$ be a matroid on $[n]$. Since $M$ is nilpotent we know by definition that $S_{M}\subsetneq M$, so we can take a point $p\notin S_{M}$. By the inductive hypothesis, we know that 
$M\setminus p$ is realizable, then applying Lemma~\ref{free}(iii) we conclude that $M$ is also realizable. 

\medskip

To show that $\Gamma_{M}$ is irreducible, we will prove that for a \textit{nilpotent} matroid $N$ on $[n]$, if $S_{N} = \emptyset$ or $\Gamma_{S_{N}}$ is irreducible, then $\Gamma_{N}$ is also irreducible. By establishing this, we can use induction on $l_{n}(M)$, the length of the nilpotent chain of $M$, to achieve the desired result. 

Let $N$ be a matroid with the given property and let $r = \rank(N)$, with $\Li$ denoting its set of subspaces. For each $l \in \Li$, let $b_{l}$ denote $\rank(S_{l})$, and let $q_{l,1}, \ldots, q_{l, b_{l}}$ be a basis of $S_{l}$. For $p\not \in S_{N}$, we also define
\[a_{p}=\begin{cases}
\rank(l) \ &\text{if} \ \Li_{p}=\{l\},\\
r \ &\text{if}\  \Li_{p}=\emptyset.
\end{cases}\]
If $S_{N}\neq \emptyset$ we define the space 
\[X\subset \Gamma_{S_{N}}\times \prod_{l\in \Li} (\CC^{r})^{\rank(l)}\times \prod_{p\notin S_{N}} \CC^{a_{p}},\]
consisting of the tuples 
\[\pare{\gamma,h_{l},\lambda_{p}:l\in \Li,p \notin S_{N}}\]
that satisfy $(h_{l})_{i}=\gamma_{q_{l,i}}$ for every $l\in \Li$ and $i\in [b_{l}]$. Here, each $h_{l}$ is considered as a tuple of $\rank(l)$ vectors in $\CC^{r}$. In particular, by definition, we have that 
$$X\cong \Gamma_{S_{N}}\times \prod_{l\in \Li}(\CC^{r})^{\rank(l)-b_{l}}\times \prod_{p\notin S_{N}} \CC^{a_{p}}.$$
If $S_{N}=\emptyset $ we define $X$ as
$\textstyle{\prod_{l\in \Li} (\CC^{r})^{\rank(l)}\times \prod_{p\notin S_{N}} \CC^{a_{p}}}.$
We know that $\Gamma_{S_{N}}$ is irreducible. Thus, in both cases, $X$ is irreducible, as it is a product of irreducible sets.

\smallskip
For any $q\notin S_{N}$, we denote $l_{q}$ for the unique subspace of $N$ containing $q$, and we define the map: 
\[
\psi: X \rightarrow (\mathbb{C}^{r})^{n} \quad \text{defined by} \quad \psi\left(\gamma, h_{l}, \lambda_{p} : l \in \Li, p \notin S_{N}\right) =
\begin{cases}
\gamma_{q} & \text{if } q \in S_{N}, \\
\textstyle{\sum_{i=1}^{\text{rank}(l_{q})} (\lambda_{q})_{i} (h_{l_{q}})_{i}} & \text{if } q \notin S_{N}.
\end{cases}
\]
We will prove that $\im(\psi)\subset V_{\mathcal{C}(N)}$. To establish this, we will show that $\psi(Y)\in V_{\mathcal{C}(N)}$, where $Y=\pare{\gamma,h_{l},\lambda_{p}:l\in \Li,p\notin S_{N}}$ is an arbitrary element of $X$. We have to prove that $\rank \{\psi(Y)_{q}:q\in l\}\leq \rank(l)$ for every $l\in \Li$. We will prove this by seeing that $\{\psi(Y)_{q}:q\in l\}\subset \overline{h_{l}}$, where $\overline{h_{l}}$ denotes the subspace associated to $h_{l}$. We consider the following cases: 
\begin{itemize}
\item Suppose $q\in S_{l}$, then $q\in \closure{\{q_{l,1},\ldots,q_{l,b_{l}}\}}$, implying that $\gamma_{q}\in \text{span}\{\gamma_{q_{l,1}},\ldots,\gamma_{q_{l,b_{l}}}\}$. Since $\gamma_{q_{l,i}}\in \overline{h_{l}}$ for every $i\in [b_{l}]$ we obtain that $\gamma_{q}\in \overline{h_{l}}$.
\item If $q\notin S_{l}$ then $l=l_{q}$ and we have that $\psi(Y)_{q}\in \overline{h_{l}}$ since $(h_{l})_{i}\in \overline{h_{l}}$ for every $i\in [\rank(l)]$.
\end{itemize}
We now prove that $\Gamma_{N} \subset \operatorname{im}(\phi)$. The idea is that every realization of $\Gamma_{N}$ can be obtained from a realization of $\Gamma_{S_{N}}$ by adding the points of $[n] \setminus S_{N}$. 
More precisely, let $\gamma \in \Gamma_{N}$. For each $l \in \Li$, take $h_{l}$ to be an extensor associated with $\gamma_{l}$. For each $q \notin S_{N}$, consider the vector $\lambda_{q} \in \CC^{\rank(l_{q})}$ as the coordinates of $\gamma_{q}$ in the basis of $\gamma_{l_{q}}$ determined by the extensor $h_{l_{q}}$. In particular, we have that: 
\[
\gamma = \psi\left(\restr{\gamma}{S_{N}}, h_{l}, \lambda_{q} : l \in \Li, q \notin S_{N}\right),
\]
which implies that $\Gamma_{N}\subset \im(\phi)$.
Since the coordinates of the vectors in the image are polynomials in the vectors of $\gamma$, $\psi$ is a polynomial map, and hence continuous. Consequently, $\operatorname{im}(\psi)$ is irreducible because $X$ is irreducible. We have also shown that $\operatorname{im}(\psi) \subset V_{\mathcal{C}(N)}$ and that $\Gamma_{N} \subset \operatorname{im}(\phi)$. Thus, $\Gamma_{N} \subset \operatorname{im}(\phi) \subset V_{\mathcal{C}(N)}$, which implies that $\Gamma_{N}$ is open in $\operatorname{im}(\phi)$ since it is open in $V_{\mathcal{C}(N)}$. Therefore, $\Gamma_{N}$ is irreducible, which completes the proof.
\end{proof}

\begin{remark}
Note that the assumption of $M$ being nilpotent in Theorem~\ref{nil rea} is essential. To establish realizability, the key ideas lie in items (i) and (ii) of Lemma~\ref{free}, which together show that if a point $p$ has degree less than two, then the realizability of 
$M\setminus \{p\}$ implies that of $M$. This inductive principle underlies the realizability of nilpotent matroids. However, we are not aware of an analogous result when the point $p$ has degree greater than one. 
Regarding the irreducibility of $\Gamma_M$, our argument clearly depends on $M$ being nilpotent. The crucial point is that when a point $p$ has degree at most one, $\Gamma_M$ can be described as a fibration over $\Gamma_{M\setminus \{p\}}$, allowing us to deduce irreducibility inductively.
\end{remark}

\subsection{Defining equations}
In this subsection, we compute the defining equations for the matroid varieties of specific nilpotent matroid families. Theorem~\ref{je} shows that for nilpotent paving matroids with no points of degree greater than two, the circuit and matroid varieties coincide. Theorem~\ref{gc 3} provides the defining equations for forest configurations. 

\subsubsection{Nilpotent paving matroids without points of degree greater than two}
In this subsection, we show that nilpotent matroids are liftable. Consequently, we show that for nilpotent paving matroids with no points of degree greater than two, the circuit variety and the matroid variety are identical. This result is established using the following proposition. We first recall the notion of a \textit{liftable matroid} from \cite{Fatemeh4}. In this subsection, let $M$ represent an $r$-paving matroid and let $[n]$ denote its ground set.

\begin{definition}\label{liftable}
Consider a collection of vectors $\gamma = \{\gamma_{p} : p \in [n]\}\subset \CC^{r}$ and a vector $q\in \CC^{r}$. Then:
\begin{itemize} 
\item  We say that  $\widetilde{\gamma}=\{\widetilde{\gamma}_{p}:p\in [n]\}\subset \CC^{r}$ is a {\em lifting} of $\gamma$ from the vector $q$ if, for each $p\in [n]$, there exists $z_{p}\in \CC$ such that $\widetilde{\gamma}_{p}=\gamma_{p}+z_{p}q$. The lifting is called {\em non-degenerate} if the vectors of $\widetilde{\gamma}$ do not all lie in the same hyperplane.
\item An $r$-paving matroid $M$ is \textit{liftable} if, for any such collection of vectors $\gamma$ 
 of rank $r-1$ in a hyperplane $H$ of $\CC^{r}$, there exists a non-degenerate lifting $\widetilde{\gamma}$, 
 within $V_{\mathcal{C}(M)}$, 
 from any vector $q \notin H$.
\end{itemize}
\end{definition}

\begin{proposition}\label{ñlh}
For a nilpotent $r$-paving matroid $M$, all of its full-rank submatroids are liftable.
\end{proposition}

\begin{proof}
We will first prove that $M$ is liftable. Let $\gamma = \{\gamma_p : p \in [n]\} \subset \mathbb{C}^r$ be a collection of vectors with rank $r-1$ in a hyperplane $H$, and let $q \notin H$. We need to show that these vectors can be lifted in a non-degenerate manner from the point $q$ to vectors in $V_{\mathcal{C}(M)}$.
We will use induction on $n$, the number of elements in the ground set of $M$, with $r$ fixed. The base case is with $n=r$, for which the result is trivial. For the inductive step, assume the result holds for all nilpotent $r$-paving matroids $N$ on $[i]$ with $i\leq n-1$ and let $M$ be an $r$-paving matroid on $[n]$. We will show it holds for $M$.
Let $L \subset \Li$ be the subset of hyperplanes $\{l \in \Li : \rank(\gamma_l) = r-1\}$. We will consider cases based on this subset.

\medskip
\textbf{Case 1.} Suppose $L = \emptyset$. Since $\rank(\gamma) = r-1$, we can select points $\{p_1, \ldots, p_{r-1}\} \subset [n]$ such that $\{\gamma_{p_1}, \ldots, \gamma_{p_{r-1}}\}$ form a basis for $H$. Consider a point $p \notin \{p_1, \ldots, p_{r-1}\}$ and lift $\gamma_p$ to a vector $\tau_p$ not in $H$. For all $p' \neq p$, set $\tau_{p'} = \gamma_{p'}$, and let $\tau$ be the resulting lifting.
Since $\tau$ consists of vectors spanning $\mathbb{C}^r$, it is clearly non-degenerate. For each $l \in \Li$, we have $\tau_{l}\subset \is{\gamma_{l},q}$, which implies $\dim(\tau_l) \leq r-1$ because $\dim(\gamma_l) \leq r-2$ (as $L = \emptyset$). Thus, $\tau \in V_{\mathcal{C}(M)}$, verifying that this lifting is valid.

\medskip
\textbf{Case 2.} Suppose $\size{L}=1$ with $L=\{l\}$. Choose a point $p\in [n]\backslash \{l\}$ and lift the vector $\gamma_{p}$ to any vector $\tau_{p}$ not in $H$. For all $p\rq \neq p$, set $\tau_{p\rq}=\gamma_{p\rq}$ and let $\tau$ be the resulting lifting. 
Since $\tau$ is formed by adding $\tau_p$ to the vectors $\gamma_{p'}$ for $p' \neq p$, and $\tau_p \notin H$ while $\tau_l = \gamma_l$, $\tau$ is non-degenerate. For each $l' \neq l \in \Li$, we have $\tau_{l\rq}\subset \is{\gamma_{l\rq},q}$, which implies $\dim(\tau_{l'}) \leq r-1$ since $\dim(\gamma_{l'}) \leq r-2$ (because $l' \notin L$). Additionally, $\tau_l = \gamma_l$, implying that $\tau \in V_{\mathcal{C}(M)}$. Thus, this lifting is valid.

\medskip
\textbf{Case 3.} Suppose $\size{L} \geq 2$. Since $M$ is nilpotent, $S_M \subsetneq M$, so there exists a point $p \in [n]$ such that $\size{\Li_p} \leq 1$. Given that $M$ has at least two dependent hyperplanes, the submatroid 
$M\setminus p$ has full rank and is also nilpotent. By induction, there exists a non-degenerate lifting of the vectors $\{\gamma_{p'} : p' \in [n]\setminus p\}$ to a collection of vectors $\{\tau_{p'} : p' \in [n]\setminus p\}$ with $\tau \in V_{\mathcal{C}(M\setminus p})$.
We now separate into cases based on the remaining hyperplanes in $L$.

\medskip
\textbf{Case 3.1.} Suppose $\size{\Li_p} = 0$. In this case, we can extend $\tau$ by defining $\tau_p = \gamma_p$. The lifting $\tau$ remains valid as it includes all necessary vectors.

\medskip
\textbf{Case 3.2.} Suppose $\size{\Li_p} = 1$ with $\Li_p = \{l\}$. If $\dim(\tau_l) \leq r-2$, we extend $\tau$ by setting $\tau_p = \gamma_p$. This lifting $\tau$ is valid since $\dim(\tau_l) \leq r-1$. If $\dim(\tau_l) = r-1$, we extend $\tau$ by defining $\tau_p$ as the intersection of the line joining $q$ and $\gamma_p$ with the hyperplane $\tau_l$. This lifting is valid because the point $p$ only belongs to the dependent hyperplane $l$.

\smallskip
Finally, since every full-rank submatroid of $M$ is a nilpotent $r$-paving matroid, it is liftable.
\end{proof}

\begin{theorem}\label{je}
Let $M$ be a nilpotent $r$-paving matroid without points of degree greater than two. Then $V_{M}=V_{\mathcal{C}(M)}$, that is the circuit variety and the matroid variety are identical.

\end{theorem}
\begin{proof}
By Proposition~\ref{ñlh}, every full-rank submatroid of $M$ is liftable. Consequently, \textup{\cite[Proposition~4.27]{Fatemeh4}} implies that $V_M = V_{\mathcal{C}(M)}$ in this case.
\end{proof}

\begin{remark}
We note that in the previous theorem, the condition that all points of $M$ have degree at most two is necessary, as it plays a crucial role in the proof of \textup{\cite[Proposition~4.27]{Fatemeh4}}.
\end{remark}


\begin{example}
Let $M$ be the 
rank-three matroid depicted in Figure~\ref{new figure} (Left). We have $S_{M} = \{1,2,3,4,5,6\}$. Additionally, in the nilpotent chain of $M$, we have that $M_{2} = \emptyset$, hence $M$ is nilpotent. Given that all points have degree at most two, Theorem~\ref{je} implies that $V_{M} = V_{\mathcal{C}(M)}$.
\end{example}

\subsubsection{Forest configurations}
In Theorem~\ref{gc 3}, we provide a complete set of defining equations for forest configurations. In \cite{sidman2021geometric}, the Grassmann-Cayley algebra was used to create polynomials in the ideals $I_M$ of matroid varieties. A key question is how to define the ideal generated by all polynomials formed through this algebra. We will precisely define this ideal, denoted 
$G_{M} \subset I_{M}$, as it has not been formally established in the literature. This definition is motivated by the following remark.

\begin{remark}\label{tñ}
Let $M$ be a matroid of rank three on $[n]$ and let $x\in [n]$. Consider distinct lines $l_{1}\neq l_{2}\in \Li_{x}$ and points $p_{1},p_{2}\in l_{1}$, $p_{3},p_{4}\in l_{2}$. Let $P\in I_{M}$.
Since $x=l_{1}\cap l_{2}$, for any realization $\gamma \in \Gamma_{M}$, we have:
\[
\gamma_{x}=\gamma_{p_{1}}\gamma_{p_{2}}\wedge \gamma_{p_{3}}\gamma_{p_{4}}=\corch{\gamma_{p_{1}},\gamma_{p_{2}},\gamma_{p_{3}}}\gamma_{p_{4}}- \corch{\gamma_{p_{1}},\gamma_{p_{2}},\gamma_{p_{4}}}\gamma_{p_{3}}.  \]
Consequently, the polynomial  obtained from $P$ by replacing the variable $x$ 
with:
\[\corch{p_{1},p_{2},p_{3}}p_{4}-\corch{p_{1},p_{2},p_{4}}p_{3}\]
is also in $I_{M}$.
\end{remark}

We will now construct an ideal for any 
rank-three matroid, using Grassmann-Cayley algebra. 

\begin{definition}
\label{gm} 
Given a 
rank-three matroid $M$, the ideal $G_{M}$ is constructed as follows:
\begin{itemize}
\item Define the set $X_{0}$ as the polynomials generating the ideal $I_{\mathcal{C}(M)}$, i.e., the brackets corresponding to the $3$-circuits of $M$. For $j \geq 1$, recursively define the set $X_{j}$ as the polynomials obtained from those in $X_{j-1}$ by modifying some of their variables, potentially leaving them unchanged, according to the procedure described in Remark~\ref{tñ}. Note that $X_{0} \subset X_{1} \subset \ldots$.
\item For $j \geq 0$, define the ideal $I_{j}$ as the ideal generated by the polynomials in $X_{j}$. Note that $I_{0} \subset I_{1} \subset \ldots$. Since the polynomial ring is Noetherian, this chain of ideals stabilizes. We denote the stabilized ideal by $G_{M}$.
\end{itemize}
\end{definition}
\begin{remark}\label{rem ñ}
If $P$ is a polynomial in $G_{M}$, then any polynomial obtained from $P$ by modifying some of its variables according to the procedure outlined in Remark~\ref{tñ} remains in $G_{M}$.
\end{remark}

It follows from Remark~\ref{tñ} and induction that $G_{M}\subset I_{M}$. The ideal $G_{M}$ is the ideal generated by all polynomials coming from the Grassmann-Cayley algebra.


\begin{example}
Let $M$ be the 
matroid of rank three depicted in Figure~\ref{new figure} (Center). We have that $\corch{6,7,8}\in I_{0}$. Since the point $6$ lies on the lines $\{3,11,6\}$ and $\{9,10,6\}$, it follows that
\[\corch{3,11,9}\corch{10,7,8}-\corch{3,11,10}\corch{4,7,8}\in I_{1}.\] Furthermore, since the point $3$ belongs to the lines $\{1,2,3\}$ and $\{4,5,3\}$, we have 
\[\corch{10,7,8}(\corch{1,2,4}\corch{5,11,9}-\corch{1,2,5}\corch{4,11,9})-\corch{4,7,8}(\corch{1,2,4}\corch{5,11,10}-\corch{1,2,5}\corch{4,11,10})\in I_{2}.\]
Moreover, 
these polynomials belong to $I_{M}$.
\begin{figure}[H]
    \centering
    \includegraphics[width=0.8\textwidth, trim=0 0 0 0, clip]{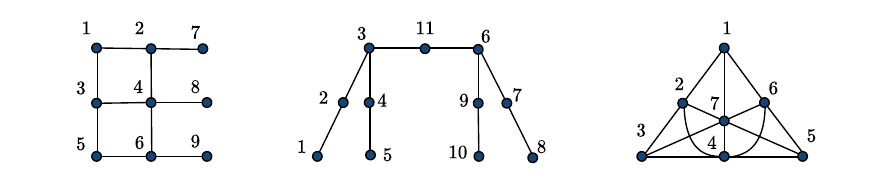}
    \caption{(Left) a nilpotent matroid; (Center) a forest configuration; (Right) Fano plane (not solvable).}
    \label{new figure}
\end{figure}
\end{example}

Consider forest 
configurations as defined in Definition~\ref{forest}. It is shown in \textup{\cite[Theorem~5.11]{clarke2021matroid}} that for forest configurations with points of degree at most two, the matroid variety and the circuit variety coincide. We now extend this analysis to configurations with points of higher degree. Here, the ideal \( G_{M} \), generated by polynomials from the Grassmann-Cayley algebra, is contained in \( I_{M} \). In Theorem~\ref{gc 3}, we will demonstrate that \( G_{M} \) and the circuit variety together define \( I_{M} \). Proving this requires establishing the following lemma.
We will first fix some notation.
\begin{notation}
For a 
rank-three matroid $M$, let $Q_{M}$ denote the set of points in $M$ with degree at least $3$ and $Q_{l}=Q_{M}\cap l$ for any line $l\in \Li_{M}$.
\end{notation}

We will need Lemma~\ref{le 1} to find the defining equations of matroid varieties associated with forest configurations. Before stating the lemma, we introduce the notion of a {\em perturbation} which refers to a motion that can be made arbitrarily small.

\begin{definition}
Consider a collection of vectors $\gamma \in \mathbb{C}^n$. Let $X$ denote a specific property. We say that a {\em perturbation} can be applied to $\gamma$ to obtain a new collection of vectors satisfying property $X$ if, for any $\epsilon > 0$, it is possible to choose, for each vector $v \in \gamma$, a vector $\widetilde{v}$ such that $\lVert v - \widetilde{v} \rVert < \epsilon$, ensuring that the new collection of vectors satisfies property $X$. 

When discussing a {\em perturbation} of a $k$-dimensional subspace $S \subset \mathbb{C}^n$, we consider $S$ as an extensor $v_1 \cdots v_k$, and apply a perturbation to $\{v_i : i \in [k]\}$ to obtain another $k$-dimensional subspace corresponding to the extensor formed by the new vectors.

The term perturbation reflects the fact that we often introduce small perturbations to a collection of vectors in order to preserve certain dependencies while eliminating others. However, a perturbation applied to a collection of vectors is simply an alteration that can be made arbitrarily small.
\end{definition}

Note that this notion of perturbation has also appeared in \cite{hocsten2004ideals}.

\begin{lemma}\label{le 1}
Let $M$ be a forest 
configuration on $[n]$. Then the following hold:
\begin{itemize}
\item[{\rm (i)}] For any $X\subset [n]$, there exists a point $x\in X$ such that $\size{\set{l\in \Li_{x}: l\cap (X\setminus \{x\})\neq \emptyset}}\leq 1.$
    \item[{\rm (ii)}] If $\gamma\in V_{\mathcal{C}(M)}$ such that $\gamma_{p}\neq 0$ for all $p\in Q_{M}$, 
    then $\gamma \in V_{M}$.
    \item[{\rm (iii)}] 
If $\gamma \in V_{\mathcal{C}(M)} \cap V(G_{M})$, then we can apply a perturbation to $\gamma$ to obtain $\tau\in \Gamma_{M}$ satisfying:
\begin{itemize}
\item For every point $p \in Q_{M}$ the vector $\tau_{p}$ is nonzero.
\end{itemize}

\end{itemize}  
\end{lemma}

\begin{proof}
(i) Suppose the contrary. Take an arbitrary point $x_{1} \in X$. By the hypothesis, there exists a line $l_{1} \in \Li_{x_{1}}$ that contains a point $x_{2} \neq x_{1} \in X$. Applying the hypothesis again, there exists a line $l_{2} \neq l_{1} \in \Li_{x_{1}}$ containing a point $x_{3} \neq x_{2} \in l_{2}$. Repeating this process, we obtain a sequence of points $x_{1}, x_{2}, \ldots$ and a sequence of lines $l_{1}, l_{2}, \ldots$ such that $x_{k} \neq x_{k+1}$, $l_{k} \neq l_{k+1}$, and $x_{k}, x_{k+1} \in l_{k}$. Since there are only a finite number of points, this process must eventually form a cycle, which contradicts the fact that $M$ is a forest.

\medskip (ii) We will prove this by showing that we can apply a perturbation to $\gamma$ to obtain $\tau\in \Gamma_{M}$, which would imply that $\gamma$ belongs to the closure $\closure{\Gamma_{M}}=V_{M}$.
We will prove this by induction on the number of lines in $M$. For the base case with no lines, it is clear that we can infinitesimally adjust the points to avoid dependencies. For the inductive step, assume the statement holds for all forest configurations with at most $m$ lines. Consider a forest $M$ with $m+1$ lines. Since $M$ is a forest, there exists a line $l$ with $\lvert S_{l} \rvert \leq 1$. Let $p \in S_{l}$. If $S_{l} = \emptyset$, choose $p$ to be any point on $l$. If $\gamma_{p} = 0$, we have two cases: if $p \in Q_{M}$, the hypothesis ensures $\gamma_{p} \neq 0$. If $\lvert \Li_{p} \rvert \leq 2$ and $\gamma_{p} = 0$, we can infinitesimally move $p$ to the intersection of the two lines containing it.

By the inductive hypothesis for $M \setminus (l\setminus \{p\})$, we can apply a perturbation to the vectors $\gamma_{q}$ for $q \in [n]\setminus (l\setminus \{p\})$ to obtain a new collection of vectors $\tau_{q}$ in the realization space of $M \setminus (l\setminus \{p\})$. Since the points $\gamma_{q}$ for $q \in l$ were on the same line and $\gamma_{p}$ was non-zero, we can infinitesimally adjust the vectors $\gamma_{q}$ for $q \in l \setminus {p}$ to align them with $\tau_{p}$ on the same line and ensure they do not lie on any other line of the configuration (as these points only belong to $l$). This yields a collection of vectors $\tau$ in $\Gamma_{M}$.

\medskip 

(iii) 
For each point $x\in [n]$ we denote by $R_{x}$ the set $\{l\in \Li_{x}:\dim(\gamma_{l})=2\}$ and consider two cases:

\medskip

\textbf{Case 1.} Suppose there exists $x \in Q_{M}$ with $\gamma_{x} = 0$ and lines $l_{1}, l_{2} \in R_{x}$ such that $\gamma_{l_{1}} \neq \gamma_{l_{2}}$. In this case, we extend $\gamma$ by defining $\tau_{x} = \gamma_{l_{1}} \cap \gamma_{l_{2}} \neq 0$ and $\tau_{p}=\gamma_{p}$ for $p\neq x$. Since $\gamma \in V(G_{M})$, all lines $\{\gamma_{l} : l \in R_{x}\}$ will contain the point $\tau_{x}$. Thus, the new collection of points is still in $V_{\mathcal{C}(M)}$ and satisfies $\tau_{x} \neq 0$. We now need to show that $\tau \in V(G_{M})$.
To prove this, consider a polynomial $r \in G_{M}$. We need to show that $r(\tau) = 0$. Choose points $p_{1}, p_{2} \in l_{1}$ and $p_{3}, p_{4} \in l_{2}$ such that
\begin{equation}\label{kh}
\tau_{x}=\gamma_{l_{1}}\wedge \gamma_{l_{2}}=\corch{\gamma_{p_{1}},\gamma_{p_{2}},\gamma_{p_{3}}}\gamma_{p_{4}}- \corch{\gamma_{p_{1}},\gamma_{p_{2}},\gamma_{p_{4}}}\gamma_{p_{3}}.
\end{equation}
Let $\overline{r}$ be the polynomial obtained from $r$ by replacing the variable $x$ with:
\[
\corch{p_{1},p_{2},p_{3}}p_{4}-\corch{p_{1},p_{2},p_{4}}p_{3}
\]
By Remark~\ref{rem ñ}, $\overline{r} \in G_{M}$. Since $\gamma \in V(G_{M})$, we have $\overline{r}(\gamma) = 0$. Using the expression for $\tau_{x}$, it follows that $r(\tau) = \overline{r}(\gamma) = 0$. 
By repeating this procedure, we will obtain a collection where no such $x$ exists. 

\medskip
\textbf{Case 2.} Suppose there does not exist $x \in Q_{M}$ with $\gamma_{x} = 0$ and lines $l_{1}, l_{2} \in R_{x}$ such that $\gamma_{l_{1}} \neq \gamma_{l_{2}}$. Let $X \subset Q_{M}$ be the set of points where $\gamma_{x} = 0$. We will prove by induction on $\lvert X \rvert$ that we can redefine the vectors $\gamma$ to non-zero vectors, resulting in a collection in $V_{\mathcal{C}(M)}$.

If $\lvert X \rvert = 0$, then $\gamma_{p} \neq 0$ for every $p \in Q_{M}$, so the result holds trivially.

For the inductive hypothesis, assume the statement holds for $\lvert X \rvert \leq m$ and let $X$ with $\lvert X \rvert = m + 1$. By applying (i), select an element $x \in X$ such that there is at most one line in $\Li_{x}$ that contains points from $X \setminus \{x\}$, and let $l_{1}$ be this line. According to the hypothesis, all lines $\{\gamma_{l} : l \in R_{x}\}$ coincide on a common line of $\mathbb{P}^{2}$, which we denote as $s$.
By the inductive hypothesis on the submatroid $M\setminus x$, we can redefine the vectors $\{\gamma_{p} : p \in X \setminus \{x\}\}$ to obtain non-zero vectors $\{\tau_{p} : p \in X \setminus \{x\}\}$ such that $\tau \in V_{\mathcal{C}(M\setminus x})$ and $\tau_{p} = \gamma_{p}$ for $p \notin X$. We then define $\tau_{x}$ as the non-zero vector in the intersection of the lines $s$ and $\tau_{l_{1}}$.
To prove that $\tau \in V_{\mathcal{C}(M)}$, we need to verify that for each line $l$, $\dim(\tau_{l}) \leq 2$. We will verify this by considering the following cases:
\begin{itemize}
\item If $l\notin \Li_{x}$ we have $\dim(\tau_{l})\leq 2$ since $\restr{\tau}{[n]\backslash \{x\}}\in V_{\mathcal{C}(M\setminus x})$.
\item If $l\in R_{x}$ we have $\dim(\tau_{l})\leq 2$ since $\tau_{x}\in s$.
\item If $l=l_{1}$ we have $\dim(\tau_{l})\leq 2$ since $\tau_{x}\in \tau_{l_{1}}$.
\item If $l \in \Li_{x} \setminus (R_{x} \cup {l_{1}})$, then by the definition of $R_{x}$, we have $\dim(\gamma_{l}) \leq 1$. Since $x$ is the only point on $l$ from $X$, it follows that $\dim(\tau_{l}) \leq 2$ in this case.
\end{itemize}
The resulting collection has no loops at the points of $Q_{M}$, thus completing the proof.
\end{proof}

Now we proceed to state and prove the main result of this section. 

\begin{theorem}\label{gc 3}
Let $M$ be a forest configuration. Then 
$I_{M}=\sqrt{I_{\mathcal{C}(M)}+G_{M}}.$
\end{theorem}

\begin{proof}
Let $I$ denote the ideal on the right. 
Since $G_{M}\subset I_{M}$, we have $I \subseteq I_{M}$. To prove the reverse inclusion, we show $V(I) \subseteq V_{M} = \closure{\Gamma_{M}}$. 
Let $\gamma$ be a collection of vectors in $V_{\mathcal{C}(M)} \cap V(G_{M})$. By Lemma~\ref{le 1}(iii), we can adjust the points $\{\gamma_{x}: x \in Q_{M} \text{ and } \gamma_{x} = 0\}$ to obtain a collection $\tau \in V_{\mathcal{C}(M)}$ that is a perturbation of $\gamma$ with $\tau_{x} \neq 0$ for all $x \in Q_{M}$. By Lemma~\ref{le 1}(ii), we can further apply a perturbation to $\tau$ to obtain 
$\kappa \in \Gamma_{M}$. Thus, $\gamma$ is in the Euclidean closure of $\Gamma_{M}$, implying $\gamma \in V_{M}$.
\end{proof}

\begin{remark}
Note that the condition of $M$ being a forest in Theorem~\ref{gc 3} is essential and cannot be weakened. This is clear in the proof of Lemma~\ref{le 1}, where the absence of cycles is crucial for establishing each of the three statements.
\end{remark}

\begin{example}
Let $M$ be the 
matroid of rank three depicted in Figure~\ref{fig:combined} (Left), which is clearly a forest configuration.  By applying Theorem~\ref{gc 3}, we find that
\[I_{M}=\sqrt{\is{\corch{1,2,7},\corch{3,4,7},\corch{5,6,7},\corch{1,2,3}\corch{4,5,6}-\corch{1,2,4}\corch{3,5,6}}}.\]
\end{example}

 \subsubsection{Defining equations}
In this subsection, 
we show that for a nilpotent matroid $M$, the dimension of liftings of a realization to a collection in $V_{\mathcal{C}(M)}$ is constant, independent of the realization or the point of lifting. We will use this result to construct polynomials in the ideal of matroid varieties in Theorem~\ref{25}.

We now extend the concept of the \textit{liftability matrix} to all matroids, whereas it was previously defined only for paving matroids in \textup{\cite{Fatemeh4}}. Recall Notation~\ref{initial notation}.

\smallskip
\begin{definition}[Liftability matrix]\label{lif mat}
Let $M$ be a matroid of rank $r$ on $[n]$, and let $q\in \mathbb{C}^{r}$. We define the \textit{liftability matrix} $\mathcal{M}_{q}(M)$ of $M$ and $q$, as the matrix with columns indexed by the points $[n]$ and rows indexed by pairs $(c,K)$, where $c$ is a circuit of $M$  
and $K\subset \textstyle \binom{[r]}{k}$, with $k$ denoting the size of $c$. The entries of this matrix are defined as follows: for each circuit $c=\{c_{1},c_{2},\ldots,c_{k}\}$ and $K\in \textstyle \binom{[r]}{k}$, the $c_{i}^{\text{th}}$ coordinate of the corresponding row is given by the polynomial 
\[ 
(-1)^{i-1}[c_{1},c_{2},\ldots,c_{i-1},\hat{c_{i}},c_{i+1},\ldots, c_{k},q]_{K},
\]
where the subindex $K$ denotes the submatrix with rows corresponding to $K$. The other entries of the row are set to $0$.
Note that the entries of the matrix are polynomials, not numbers. For a collection of vectors $\gamma=\{\gamma_{p}:p\in [n]\}$ of $\CC^{r}$, we denote $\mathcal{M}_{q}^{\gamma}(M)$ as the matrix obtained by evaluating the entries of $\mathcal{M}_{q}(M)$ at the vectors of $\gamma$, i.e.~substituting the vector of $r$ indeterminates associated with each $p\in [n]$ by the vector $\gamma_{p}$.
\end{definition}

\begin{example}
Consider the quadrilateral set $QS$, illustrated in Figure~\ref{fig:combined} (Center). The set of points of $QS$ is $[6]$, and its circuits of size $3$ are given by $\set{\{1,2,3\},\{1,5,6\},\{2,4,6\},\{3,4,5\}}$. Thus, for any $q\in \CC^{3}$, its liftability matrix is the following:
\begin{equation}\label{matrix quad}\mathcal{M}_{q}(QS)=\begin{pmatrix}
\corch{2,3,q} & -\corch{1,3,q} & \corch{1,2,q} &0 &0 &0 \\
\corch{5,6,q} & 0 &0 &0 &-\corch{1,6,q} &\corch{1,5,q} \\ 
 0  &\corch{4,6,q} &0 &-\corch{2,6,q} & 0&\corch{2,4,q}    \\
 0& 0 & \corch{4,5,q} &-\corch{3,5,q}& \corch{3,4,q}& 0
\end{pmatrix}.
\end{equation}
\end{example}

For the remainder of this subsection, $M$ represents a rank $r$ matroid on $[n]$.
We have the following relation between the liftability matrix and the liftings of vectors. 

\begin{lemma}\label{lift eq}
Let $\gamma \in V_{\mathcal{C}(M)}$ be a collection of vectors in $\mathbb{C}^{r}$ and $q \in \mathbb{C}^{r}$. The following statements are equivalent: 
\begin{itemize}
    \item A vector $(z_{p})_{p \in [n]} \in \mathbb{C}^{n}$ belongs to the kernel of $\mathcal{M}_{q}^{\gamma}(M)$.
    \item The lifted vectors
$\delta = \{\delta_{p} = \gamma_{p} + z_{p} \cdot q : p \in [n]\}$
are in $V_{\mathcal{C}(M)}$.
\end{itemize}
\end{lemma}
\begin{proof}
The collection $\delta$ belongs to $V_{\mathcal{C}(M)}$ if and only if the vectors 
 $\{\gamma_{c_{i}} + z_{c_{i}} \cdot q: i\in [k]\}$ are dependent for every circuit $c = \{c_{1}, \ldots, c_{k}\} \in \mathcal{C}(M)$ with $k \leq r$. This is equivalent to all the determinants of the $k \times k$ submatrices of the matrix formed by these vectors being zero.
Consider the $k \times k$ submatrix corresponding to a subset $K \in \textstyle \binom{[r]}{k}$. The determinant of this submatrix is:
\[
\det(\gamma_{c_{1}}+z_{1}\cdot q,\ldots,\gamma_{c_{k}}+z_{k}\cdot q)_{K}=
\sum_{i=1}^{k} z_{i}(-1)^{k-i}\det(\gamma_{c_{1}},\ldots,\hat{\gamma}_{c_{i}},\ldots, \gamma_{c_{k}},q)_{K}=0.
\]
The first equality holds because $\gamma \in V_{\mathcal{C}(M)}$. The coefficients multiplying $z_{c_{i}}$ in this expression are exactly the non-zero entries of the row corresponding to $(c, K)$ in the matrix $\mathcal{M}_{q}^{\gamma}(M)$ (after multiplying them by $(-1)^{k-1}$). Hence, the collection of vectors $\delta$ is in $V_{\mathcal{C}(M)}$ if and only if the vector $z = \{z_{p}: p \in [n]\}$ is in the kernel of the matrix $\mathcal{M}_{q}^{\gamma}(M)$.
\end{proof}

\begin{definition}
For a collection of vectors $\gamma\in V_{\mathcal{C}(M)}$ in $\CC^{r\rq}$ with $r\rq\geq r$ and a vector $q\in \CC^{r\rq}$ we will denote by $\dim_{q}\pare{\gamma}$ the dimension of the liftings of the vectors of $\gamma$ to a collection in $V_{\mathcal{C}(M)}$ from the point $q$. Using  Lemma~\ref{lift eq} we know that in the case of $r\rq=r$ this number coincides with the dimension of the kernel of the matrix $\mathcal{M}_{q}^{\gamma}\pare{M}$.  
\end{definition}

We now show that Lemma~\ref{lift eq} does not hold when the dimension of the ambient space differs from the rank of $M$.

\begin{lemma}\normalfont\label{22}
Let $\gamma\in \Gamma_{M}$ in $\CC^{r\rq}$ with $r\rq> r$ and $q\in \CC^{r\rq}\setminus \langle\gamma_{p}:p\in [n]\rangle$. Then $\dim_{q}(\gamma)=r$.
\end{lemma}

\begin{proof}
Let $V=\langle\gamma_{p}:p\in [n]\rangle$ and suppose that $\gamma_{q_{1}},\ldots,\gamma_{q_{r}}$ is a basis of $V$. Consider a lifting $\{\tau_{p}=\gamma_{p}+z_{p}\cdot q:p\in [n]\}$ of $\gamma$. For any choice of the numbers $\{z_{q_{1}},\ldots,z_{q_{r}}\}$,  consider $H$ as the subspace of dimension $r$ generated by the vectors 
$\{\tau_{q_{i}}:i\in [r]\}$. 

\medskip

Note that $q \notin H$, and $H$ is a hyperplane in the subspace $V + \langle q \rangle$. For $\tau$ to be in $V_{\mathcal{C}(M)}$, each vector $\tau_{p}$ must be the projection of $\gamma_{p}$ onto $H$ from the vector $q$. Conversely, any such lifting belongs to $V_{\mathcal{C}(M)}$ because the vectors in $\tau$ are derived by projecting the vectors in $\gamma$ onto the hyperplane $H$ in the subspace $V + \langle q \rangle$.
Therefore, since the coefficients $z_{q_{1}}, \ldots, z_{q_{r}}$ were chosen arbitrarily, we conclude that the dimension of the kernel is exactly $r$, as required.
\end{proof}

\begin{notation}
Let $M$ be a matroid on $[n]$. We denote by $M(0)$ the set of points $\{p\in [n]:\Li_{p}=\emptyset\}$.
\end{notation}

\begin{lemma}\label{23}
For any collection $\gamma \in \Gamma_{M}$ 
and $q\in\CC^{r}$ not in the span of $\gamma$, we have: 
\[\dim_{q}(\gamma)=\dim_{q}(\restr{\gamma}{S_{M}})+\sum_{l\in \Li}\rank(l)-\rank(S_{l})+\size{M(0)}.\]

\end{lemma}

\begin{proof}
 For each subspace $l\in \Li$ denote by $b_{l}$ the number $\rank(S_{l})$ and let $p_{l_{1}},\ldots,p_{l_{b_{l}}}$ be a basis of $S_{l}$.\\
Consider any lifting $\delta\in V_{\mathcal{C}(S_{M})}$ of the vectors $\restr{\gamma}{S_{M}}$. For each $l\in \Li$ consider points $r_{l_{1}},\ldots,r_{l_{\rank(l)-b_{l}}}$ of $l\setminus S_{l}$. It is clear that the points $\{p_{l_{1}},\ldots,p_{l_{b_{l}}},r_{l_{1}},\ldots,r_{l_{\rank(l)-b_{l}}}\}$ form a basis of $\closure{l}$. We extend $\delta$ to the points $r_{l_{i}}$ by considering an arbitrary lifting of the vectors $\gamma_{r_{l_{i}}}$. For ease of notation, we will denote the lifted vectors by $\delta$. Observe that for every $l\in \Li$, the vectors 
$\{\gamma_{p_{l_{1}}},\ldots,\gamma_{p_{l_{b_{l}}}},\gamma_{r_{l_{1}}},\ldots,\gamma_{r_{l_{\rank(l)-b_{l}}}}\}$
are independent. 
Therefore, the vectors 
$\delta_{p_{l_{1}}},\ldots,\delta_{p_{l_{b_{l}}}},\delta_{r_{l_{1}}},\ldots,\delta_{r_{l_{\rank(l)-b_{l}}}}$
are also independent, since they are obtained from a lifting of the vectors of $\gamma$ from the point $q$ outside $\gamma$. Hence, to extend the lifting $\delta$ to the remaining points of $M$, 
for 
each of the remaining points $p\in l$, 
$\gamma_{p}$ has to be lifted to the vector
$$\is{\gamma_{p},q}\cap \is{\delta_{p_{l_{1}}},\ldots,\delta_{p_{l_{b_{l}}}},\delta_{r_{l_{1}}},\ldots,\delta_{r_{l_{\rank(l)-b_{l}}}}}.$$
Since all these vectors lie in the $\rank(l)+1$ dimensional space $\is{\gamma_{l},q}$ and $\is{\delta_{p_{l_{1}}},\ldots,\delta_{p_{l_{b_{l}}}},\delta_{r_{l_{1}}},\ldots,\delta_{r_{l_{\rank(l)-b_{l}}}}}$ forms a hyperplane within this space, the intersection is non-empty. Consequently, any lifting of $\gamma$ to $V_{\mathcal{C}(M)}$ is fully determined by first lifting $\restr{\gamma}{S_{M}}$ to a collection in $V_{\mathcal{C}(S_{M})}$, then choosing arbitrary liftings for the $\rank(l) - b_{l}$ points $r_{l_{1}}, \ldots, r_{l_{\rank(l) - b_{l}}}$ for each $l \in \Li$, and finally selecting arbitrary liftings for the points corresponding to $M(0)$. This process ensures that the desired equality holds.
\end{proof}

Using Lemmas~\ref{22} and~\ref{23}, we establish the following proposition that for a nilpotent matroid $M$, the dimension of the liftings of a realization of $M$ to a collection in the circuit variety is constant. This dimension is independent of both the specific realization of $M$ and the point of lifting.

\begin{proposition}\label{24}
Let $M$ be a nilpotent matroid, and let $\gamma \in \Gamma_{M}$ be a collection of vectors in $\mathbb{C}^{r}$. For a point $q \in \mathbb{C}^{r}$ not in the span of $\gamma$, the dimension $\dim_{q}(\gamma)$ is independent of both the choice of $q$ and the realization $\gamma$. We denote this fixed dimension by $\dim(M)$.
\end{proposition}

\begin{proof}
Let 
\[M_{0}\supset M_{1}\supset \cdots\]
be the nilpotent chain of $M$  and let $M_{k}$ be the first subset of the chain not of full rank. Applying Lemma~\ref{23} we have that $\textstyle{\dim_{q}(\restr{\gamma}{M_{j-1}})=\dim_{q}(\restr{\gamma}{M_{j}})+c_{j}}$
for every $j\in [k]$, where $c_{j}$ is the constant 
$\textstyle{c_j= \sum_{l\in \Li_{M_{j-1}}}\rank(l)-\rank(S_{l})+\size{M_{j-1}(0)}}.$
Applying Lemma~\ref{22} we also have that $\dim_{q}(\restr{\gamma}{M_{k}})=\rank(M_{k}).$ Therefore, we have that
\begin{equation}\label{ktc}
\dim_{q}(\gamma)=\sum_{j=1}^{k}c_{j}+\rank(M_{k}),
\end{equation}
which is a constant independent of both $q$ and $\gamma$. This completes the proof.
\end{proof}

Proposition~\ref{24} 
provides a recursive formula for computing $\dim(M)$ for a nilpotent matroid.

\begin{example}\label{ñr}
Let $M$ be the $4$-paving matroid on the ground set $[11]$ with dependent hyperplanes
\[\Li=\{\{1,2,3,4\},\{1,2,5,6\},\{1,3,5,7\},\{1,5,4,8\},\{2,3,5,9\},\{2,6,10,11\}\}.
\]
We have the following chain of submatroids \[M_{0}=M,\ M_{1}=\{1,2,3,4,5,6\},\ M_{2}=\{1,2\},\ M_{3}=\emptyset.\]
Let~$c_{1}$ and $c_{2}$ denote the constants from \eqref{ktc}. Since $M_{1}$ has dependent hyperplanes $\Li_{M_{1}}=\{\{1,2,3,4\},\{1,2,5,6\}\}$ and  $S_{M_{1}}=\{1,2\}$ we have that $c_{2}=2$. Similarly, it is straightforward to see that $c_{1}=1$. We also have that $\rank(M_{2})=2$, and putting all of this together, we get that $\dim(M)=c_{1}+c_{2}+\rank(M_{2})=5$.
\end{example}

Using Proposition~\ref{24} we have the following theorem for nilpotent matroids.

\begin{theorem}\label{25}
Let $M$ be a nilpotent matroid of rank $r$ on $[n]$. Then, the $n-\dim(M)+1$ minors of the matrix $\mathcal{M}_{q}(M)$ lie in $I_{M}$ for any $q\in \CC^{r}$.
\end{theorem}

\begin{proof}
Let $\gamma \in \Gamma_{M}$. Applying Proposition~\ref{24}, we have that 
$\dim(\ker(\mathcal{M}_{q}^{\gamma}(M)))=\dim(M)$
for any vector $q$ outside $\gamma$.  Consequently for these vectors $q$, we have $\rank(\mathcal{M}_{q}^{\gamma}(M))=n-\dim(M)$, implying that the $n-\dim(M)+1$ minors of the matrix $\mathcal{M}_{q}(M)$ vanish at $\gamma$. This holds for any vector $q$ outside $\gamma$ and since the set of vectors $q$ outside $\gamma$ is a dense subset of $\CC^{n}$, we get that the the $n-\dim(M)+1$ minors of the matrix $\mathcal{M}_{q}(M)$ vanish at $\gamma$ for every $q\in \CC^{r}$.
Since $\gamma$ is an arbitrary realization of $M$ we get the desired result.  
\end{proof}

\section{Solvable matroids}\label{solv}
In this section, we introduce a new class of matroids, which we refer to as 
\textit{solvable matroids}. 
Our main result is Theorem~\ref{irreducible}, which establishes the irreducibility of realization spaces of paving matroids that have no points of degree greater than two. Additionally, we prove the irreducibility of the associated varieties of the solvable matroids of rank three.

\subsection{Definition and basic properties}\label{solv 1}

We will first fix our notation throughout this subsection. Let $M$ be a matroid of rank $r$ on $[n]$.

\begin{notation}\label{m6}
\begin{itemize}
    \item For a point $p\in [n]$ we define  \begin{equation}\label{am}
a_{p}=\sum_{l\in \Li_{p}} \ \rank(l) -r(\size{\Li_{p}}-1).
\end{equation}
We define $Q_{M}=\{p\in [n]:\ a_{p}\leq 0\}$  
and 
$Q_{l}=Q_{M}\cap l$  
for any subspace $l\in \Li_M$. 

\item For any $\gamma \in V_{\mathcal{C}(M)}$, we denote the subspace 
$\textstyle\bigcap_{l \in \Li_{p}}\gamma_{l}$ by $V_{\gamma,p}$.
\end{itemize}
\end{notation}

The motivation behind defining $a_p$ in \eqref{m6} is that $Q_{M}$ contains all the points $p \in [n]$ for which, given any realization of $M$, we can \textit{expect} the intersection of the subspaces associated with the elements of $\Li_{p}$ to have a specific dimension. This expected dimension is denoted by $a_{p}$.


\begin{remark}
\begin{itemize}
\item If $M$ is an $r$-paving matroid, then $a_{p} = r - \left|\Li_{p}\right|$. Thus, $p \in Q_{M}$ if and only if $\left|\Li_{p}\right| \geq r$. In particular, for a 
matroid of rank three, $p \in Q_{M}$ if and only if three lines of $\Li$ contain this point.
\item From Definition~\ref{m6}, it follows that $Q_{N} \subset Q_{M}$ for any submatroid $N \subset M$. Note that when considering $Q_{N}$ for a submatroid $N$, Equation~\eqref{am} uses the number $\rank(N)$ in place of $\rank(M)$.
\end{itemize}
\end{remark}

We now introduce the family of \textit{solvable matroids}.

\begin{definition}[Solvable matroid]\normalfont\label{m7}
Let $M$ be a matroid. We define the \textit{solvable chain} of $M$ as the following chain of submatroids of $M$: 
$$M^{0}=M,\quad  M^{1}=Q_{M}, \ \text{and for every $j\geq 1,$}\quad  M^{j+1}=Q_{M^{j}}.$$
We call $M$ \textit{solvable} if $M^{j}=\emptyset$ for some $j$. In this case, we denote the length of the chain by $l_{s}(M)$.
\end{definition}

Note that in the preceding definition, we are identifying each subset of points $M^{j}$ with the submatroid it defines.

\begin{example}
The rank-three matroids 
depicted in Figure~\ref{fig:combined} are solvable. 
The configuration in the center has a solvability chain of length 1, while the other two configurations each have a solvability chain of length 2. An example of a non-solvable matroid is the Fano plane, illustrated in Figure~\ref{new figure} (Right).
The Fano plane is not solvable because all of its points have degree $3$.
\end{example}

We also obtain a characterization of solvable matroids using the same argument as in Lemma~\ref{equiv 2}.

\begin{lemma}\label{equiv}
Let $M$ be a matroid. Then the following hold:
\begin{itemize}
    \item[{\rm (i)}] $M$ is not \textit{solvable} if and only if there exists a submatroid $\emptyset \subsetneq N \subset M$ with $Q_{N}=N$. 
    \item[{\rm (ii)}] Assume that $M$ is an $r$-paving matroid. Then $M$ is \textit{solvable} if and only if every submatroid $N$ of full rank, with at least one dependent hyperplane, has a dependent hyperplane $l$ with $\size{Q_{l}}\leq r-1$.
\end{itemize} 
\end{lemma}

\begin{example}\label{ktu}
For an $r$-paving matroid $M$, $Q_{M}$ is the submatroid formed by the points of degree at least $r$, and $M$ is solvable if and only if the sequence obtained by recursively taking the submatroid formed by points of degree at least $r$ eventually terminates in the empty set. 
By Lemma~\ref{equiv}, $M$ 
is solvable if and only if there is no submatroid $N$ such that all its points have degree at least~$r$.
\end{example}

\begin{example}\label{ktu 2}
In particular, by Example~\ref{ktu}, we know that any paving matroid in which all points have degree at most two is solvable, as in this case $Q_{M} = \emptyset$.
\end{example}

\begin{remark}\label{ktv}
Example~\ref{ktu} illustrates that the class of solvable matroids is notably broad. It is conjectured in \cite{mayhew2011asymptotic} that asymptotically, almost all matroids are paving. For paving matroids, the solvable ones are those that do not contain a full-rank submatroid where every point has a degree at least equal to the rank of the matroid. For example, for matroids of rank three, nearly all instances studied in the literature are solvable configurations.
\end{remark}

\begin{notation}\label{ño}
For an $r$-paving matroid $M$ on $[n]$ and a dependent hyperplane $l\in \Li$ with $\size{Q_{l}}\leq r-1$ we use the notation $M- l$, with slight abuse of notation,  to denote the submatroid with points $([n]\backslash l)\cup (Q_{l})$ and we say that $M- l$ is the submatroid obtained by deleting the dependent hyperplane $l$. 
We denote by $M- \{l_{1},\ldots,l_{k}\}$ the submatroid obtained by deleting the dependent hyperplanes $l_{1},\ldots,l_{k}$ in that order, where $l_{j}$ is a dependent hyperplane of $M-\{l_{1},\ldots,l_{j-1}\}$ for $j\in [k]$.
\end{notation}

Note that in Notation~\ref{ño}, we reuse the notation from~\ref{kb}, which was previously used for nilpotent matroids. However, the context will make it clear, as we will now exclusively consider solvable matroids. Using Lemma~\ref{equiv}(ii), we obtain the following corollary, whose proof is similar to that of Corollary~\ref{rel}.

\begin{corollary}\label{rel 6}
Let $M$ be an $r$-paving matroid. Then $M$ is solvable if and only if there exists a sequence of dependent hyperplanes $l_{1},\ldots,l_{k}$ such that:
\begin{itemize}
    \item $M- \set{l_{1},\ldots,l_{k}}$ does not contain any dependent hyperplane, and 
\item $\lvert{l_{j}\cap Q_{M-\set{l_{1},\ldots,l_{j-1}}}\rvert} \leq r-1$
for every $j\in [k]$.
\end{itemize}
\end{corollary}

 Applying Corollary~\ref{rel 6}, we have that: 
\begin{corollary}\label{cors}
A 
rank-three matroid $M$ is \textit{solvable} if and only if there is a sequence of lines $l_{1},\ldots,l_{k}$ such that:
\begin{itemize}
    \item $M- \set{l_{1},\ldots,l_{k}}$ contains no line, and 
    \item 
$\lvert{l_{j}\cap Q_{M-\set{l_{1},\ldots,l_{j-1}}}\rvert} \leq 2$
for every $j\in [k]$. 
\end{itemize}
\end{corollary}

We now recall a theorem from \cite{clarke2021matroid}.
\begin{theorem}[\textup{\cite[Theorem~4.5]{clarke2021matroid}}]\label{rem irred 3}
Let $M$ be a rank-three matroid on $[n]$ and $l\in \Li$ a line. 
If $\size{Q_{l}}\leq 2$ and 
$\Gamma_{M-l}$ is irreducible, then $\Gamma_{M}$ is also irreducible. 
\end{theorem}

Thus, we can derive the following result.

\begin{theorem}\label{teosol}
Let $M$ be a solvable 
matroid of rank three. Then $\Gamma_{M}$ is irreducible.
\end{theorem}
\begin{proof}
The result follows from applying Corollary~\ref{cors}
together with Theorem~\ref{rem irred 3} and induction.
\end{proof}

\begin{remark}
We now provide some motivation for why the assumption of $M$ being solvable is necessary in Theorem~\ref{teosol}. The key idea relies on examining the role of points with degree at most two. Specifically, if $p$ is a point of degree at most two in a rank-three matroid $M$, the irreducibility of $\Gamma_{M\setminus \{p\}}$ implies the irreducibility of $\Gamma_{M}$. This follows from the fact that $\Gamma_{M}$ can be described as a fibration over $\Gamma_{M\setminus \{p\}}$. The fibration can be interpreted in the following cases:
\begin{itemize}
\item If $\text{deg}(p)=0$, then $\Gamma_{M}$ is constructed from $\Gamma_{M\setminus \{p\}}$ along with the selection of a generic vector in $\CC^{3}$ corresponding to $p$.
\item If $\text{deg}(p)=1$ with $\Li_{p}=\{l\}$, then $\Gamma_{M}$ is constructed from $\Gamma_{M\setminus \{p\}}$ and the selection of a generic vector in the two-dimensional subspace corresponding to $l$.
\item If $\text{deg}(p)=2$ with $\Li_{p}=\{l_{1},l_{2}\}$, then $\Gamma_{M}$ is constructed from $\Gamma_{M\setminus \{p\}}$ and the selection of a generic vector in the one-dimensional subspace corresponding to the intersection of the subspaces associated with $l_{1}$ and $l_{2}$.
\end{itemize}
We also emphasize that examining the points lying on at least three lines is a natural consideration, one that has been adopted in several other works, such as \cite{nazir2012connectivity, corey2023singular, guerville2023connectivity, clarke2021matroid}.
\end{remark}

Observe that Theorem~\ref{teosol} simplifies the process of identifying matroids of rank three with irreducible realization spaces; we need only verify that the solvability chain terminates at the empty set. Now we will see that every nilpotent matroid is solvable.

\begin{lemma}\label{k3}
If $M$ is a nilpotent matroid, then  $M$ is also solvable.
\end{lemma}

\begin{proof}
We use induction to prove that $M^j \subseteq M_j$ for every $j \geq 0$.
For $j = 0$, we have $M^0 = M$ and $M_0 = M$. Thus, $M^0 \subseteq M_0$ holds.
For the inductive hypothesis suppose that $M^{j}\subset M_{j}$, then \[M^{j+1}=Q_{M^{j}}\subset Q_{M_{j}}\subset S_{M_{j}}=M_{j+1}.\] 
Since $M$ is nilpotent, there exists an index $k$ such that $M_k = \emptyset$, and so $M^k = \emptyset$, hence  
$M$ is solvable.
\end{proof}

\subsection{Paving matroids without points of degree greater than two}\label{solv 3}


The main result of this subsection is Theorem~\ref{irreducible}, which establishes that paving matroids without points of degree greater than two possess irreducible realization spaces. This generalizes \textup{\cite[Theorem 4.5]{clarke2021matroid}}, where a similar conclusion was reached for matroids of rank three. Throughout this section, we assume that $M$ is an $r$-paving matroid with no points of degree greater than
two on $[n]$. We will avoid repeating this assumption in the statements of lemmas.

\medskip

Our goal now is to prove Theorem~\ref{irreducible}, for which we first need to establish the following lemmas.

\begin{lemma}\label{basis}
Let $S$ and $T$ be subspaces of $\mathbb{C}^r$ with dimensions $k$ and $r-k$, respectively, such that $S \cap T = \{0\}$. Let $v_1, \ldots, v_k$ be vectors with $\is{v_{1},\ldots,v_{k}}+T=\CC^{r}$. For each $i \in [k]$, let $t_i$ be an extensor associated with the subspace $T + \langle v_i \rangle$. Let $s$ be an extensor associated with $S$. Then, the vectors
$\{t_i \wedge s : i \in [k]\}$
form a basis for $S$.
\end{lemma}

\begin{proof}
Since $\dim(T + \langle v_i \rangle) + \dim(S) = r + 1$ and $T + S = \mathbb{C}^r$, we know that $t_i \wedge s$ is a non-zero vector, which we denote as $s_i$. This vector $s_i$ belongs to $T + \langle v_i \rangle \cap S$, so we can write $s_i = w_i + \lambda_i v_i$ with $w_i \in T$. Note that $\lambda_i \neq 0$ because $S \cap T = \{0\}$. Furthermore, since $\langle v_1, \ldots, v_k \rangle \cap T = \{0\}$, the vectors $\{s_i : i \in [k]\}$ are linearly independent. Thus, they form a basis for $S$.
\end{proof}

\begin{definition}\normalfont
Denote the subspace $\langle e_{r-1},e_{r} \rangle\subset \CC^{r}$ by $E$.
For each $1\leq i\leq r-2$, denote the subspace $E+ \langle e_{i} \rangle$ by $E_{i}$. We say that $\gamma \in \Gamma_{M}$ is \textit{normal} if $V_{\gamma,p} \cap E = \{0\}$ for every point $p\in [n]$ of degree two. We denote by $\Gamma_{M}^{1}$ the subset of $\Gamma_{M}$ consisting of these \textit{normal} collections.
\end{definition}

\begin{lemma}\label{gamma 1}
We have that $\Gamma_{M}^{1}$ is dense in $\Gamma_{M}$.
\end{lemma}

\begin{proof}
We will prove the lemma by seeing that any $\gamma \in \Gamma_{M}$ can be infinitesimally perturbed to a collection in $\Gamma_{M}^{1}$.
If $\gamma$ is normal, then the assertion holds. 
Otherwise, assume there is a point $p \in [n]$ of degree two such that $V_{\gamma,p} \cap E \neq \{0\}$. Let $T$ be a projective~transformation~infinitesimally close to $\Id$ with $T(V_{\gamma,p}) \cap E = \{0\}$. Then, $T(\gamma)$ represents a perturbation of $\gamma$,~and 
$$V_{T(\gamma),p} \cap E = T(V_{\gamma,p}) \cap E = \{0\}.$$
Repeating this process, we will eventually obtain a normal $\kappa \in \Gamma_{M}^{1}$ after a finite number of steps.
\end{proof}

We now establish the main result of this subsection.

\begin{theorem}\label{irreducible}
 If $M$ be a paving matroid without points of degree greater than two, then $\Gamma_{M}$ is irreducible.
\end{theorem}

\begin{proof}

Recall that $M$ has rank $r$ and ground set $[n]$ with $\Li$ denoting its set of dependent hyperplanes. Define the space $X$ as
\[\textstyle{\prod_{l \in \Li} (\mathbb{C}^{r})^{r-1} \times \prod_{p\in [n]} \mathbb{C}^{a_{p}}.}\]

The variety $X$ is clearly irreducible. For each tuple $\left( h_{l}, \lambda_{p} : l \in \Li, p\in [n]\right) \in X$ and each $q\in S_{M}$, we have that
\[\textstyle\bigwedge_{l \in \Li_{q}} h_{l} \wedge E_{i}\]
is a vector for every $i \in [r-2]$. We denote this vector by $k_{q,i}$. Note that $k_{q,i} \in \overline{h_{l}}$ for every $l \in \Li_{q}$, where $\overline{h_{l}}$ denotes the subspace associated with the extensor $h_{l}$.

Now we define the map $\psi: X \rightarrow (\CC^{r})^{n}$ as follows: 
\[
\psi\left(h_{l}, \lambda_{p} : l \in \Li, p \in [n]\right)_{q} =
\begin{cases}
\lambda_{q} & \text{if } \Li_{q}=\emptyset, \\ 
\textstyle{\sum_{i=1}^{r-1}\lambda_{q,i}(h_{l})_{i}} & \text{if $\Li_{q}=\{l\}$}.\\
\textstyle{\sum_{i=1}^{r-2} \lambda_{q,i} k_{q,i}} & \text{if } \size{\Li_{q}}=2.
\end{cases}
\]
We will show that $\im(\psi) \subset V_{\mathcal{C}(M)}$. To establish this, we prove that $\psi(Y) \in V_{\mathcal{C}(M)}$, where $Y = (h_{l}, \lambda_{p} : l \in \Li, p\in [n])$ is an arbitrary element of $X$. Specifically, we  prove that $\rank \{\psi(Y)_{q} : q \in l\} \leq r-1$ for every $l \in \Li$. We will do this by showing that $\{\psi(Y)_{q} : q \in l\} \subset \overline{h_{l}}$. We will proceed by considering the following cases:
\begin{itemize}
\item Suppose $\Li_{q}=\{l\}$, then $\psi(Y)_{q} \in \overline{h_{l}}$, because $(h_{l})_{i} \in \overline{h_{l}}$ for every $i \in [r-1]$.
\item If $\size{\Li_{q}}=2$, then $\psi(Y)_{q} \in \overline{h_{l}}$, because $k_{q,i} \in \overline{h_{l}}$ for every $i \in [r-2]$.
\end{itemize}
We will now prove that $\Gamma_{M}^{1} \subset \im(\phi)$. 
Let $\gamma \in \Gamma_{M}^{1}$. For each $q \in S_{M}$, we have $V_{\gamma, q} \cap E = \{0\}$. By Lemma~\ref{basis}, we deduce that the vectors
$\textstyle{\{\bigwedge_{l \in \Li_{q}} h_{l} \wedge E_{i} : i \in [r-2]\}}$
form a basis of $V_{\gamma, q}$, where $h_{l}$ denotes an extensor associated with $\gamma_{l}$. Since $\gamma_{q} \in V_{\gamma, q}$, we can represent $\gamma_{q}$ by the vector $\lambda_{q} \in \CC^{r-2}$ in this basis. It follows that
$\gamma = \psi(h_{l}, \lambda_{q} : l \in \Li, q\in [n])$.
This proves that $\Gamma_{M}^{1} \subset \im(\phi)$.

Since the coordinates of the vectors $k_{q_{i}}$ are polynomials in the vectors of $\gamma$, the map $\psi$ is a polynomial map and hence continuous. Consequently, $\im(\psi)$ is irreducible because $X$ is irreducible. We have also established that $\im(\psi) \subset V_{\mathcal{C}(M)}$ and that $\Gamma_{M}^{1} \subset \im(\phi)$. Therefore, we obtain $\Gamma_{M}^{1} \subset \im(\phi) \subset V_{\mathcal{C}(M)}$, which implies that $\Gamma_{M}^{1}$ is open in $\im(\phi)$, as $\Gamma_{M}^{1}$ is open in $V_{\mathcal{C}(M)}$. Hence, $\Gamma_{M}^{1}$ is irreducible. By Lemma~\ref{gamma 1}, this implies that $\Gamma_{M}$ is also irreducible, completing the proof.
\end{proof}

\section{Applications}\label{applications}

In this section, we describe the motivation and applications of nilpotent and solvable matroids, showing the contexts in which they naturally arise. In particular, we discuss the relevance of studying each of these specialized classes of matroids.  We first introduce key definitions and questions.

\subsection{Hypergraph varieties}

For an $d\times n$ matrix $X=\pare{x_{i,j}}$ of indeterminates, we denote by $X_{F}$ the submatrix of $X$ with columns indexed by $F\subset [n]$. We also recall Notation~\ref{notation 1}.

\begin{definition}\label{def hyp}
A hypergraph $\Delta$ on the vertex set $[n]$ is a subset of the power set $2^{[n]}$. We assume that no proper subset of an element of $\Delta$ is in $\Delta$. The elements of $\Delta$ are refered to as edges.
\begin{itemize}
\item The {\em determinantal hypergraph ideal} of $\Delta$ is
\[I_{\Delta}=\is{\corch{A|B}_{X}:A\subset [d],B\in \Delta,\size{A}=\size{B}}\subset \CC[X].\]
\item The associated variety, which is the zero set of $I_{\Delta}$, is given by
\[V_{\Delta}=\{Y\in \CC^{d\times n}:\rank(Y_{F})<\size{F} \ \text{for each $F\in \Delta$}\}.\]
\end{itemize}
Note that both $I_{\Delta}$ and $V_{\Delta}$ depend on $d$, the dimension of the ambient space. However, for simplicity, we omit $d$ from the notation. 
\end{definition}

The circuit variety of a matroid $M$ is a hypergraph variety, of the hypergraph whose edges correspond to the circuits of $M$. We have the following natural question, studied in \cite{pfister2019primary,clarke2021matroid,clarke2020conditional,liwski2025minimal}.

\begin{question}\label{quest 1}
Determine the irreducible decomposition of the circuit variety $V_{\mathcal{C}(M)}$ or, more generally, of hypergraph varieties $V_{\Delta}$.
\end{question}

In \cite{clarke2021matroid}, a decomposition strategy was proposed by identifying the set of minimally dependent matroids associated with $M$. We now introduce the notion of minimal matroids.

\begin{definition}\label{dependency}
Let $N_{1}$ and $N_{2}$ be matroids on $[n]$. We write $N_{1}\leq N_{2}$ if $\mathcal{D}(N_{1})\subset \mathcal{D}(N_{2})$. This defines a partial order on matroids, known as the \emph{dependency order}. 
\end{definition}

\begin{definition}\label{def min}
The collection of minimal matroids for a matroid $M$ is defined as: \[\min(M)=\min\{N:N>M\}.\] 
\end{definition}

In \cite{liwski2025minimal}, an algorithm for identifying minimal matroids is presented, along with a decomposition strategy for circuit varieties. The key ingredient of this approach is the following proposition, which provides a potential non-redundant decomposition of $V_{\mathcal{C}(M)}$.

\begin{proposition}[\textup{\cite[Proposition~4.1]{liwski2025minimal}}]\label{deco circ}
Let $M$ be a matroid. Then
$V_{\mathcal{C}(M)}=\bigcup_{N\in \min(M)}V_{\mathcal{C}(N)}\ \cup \ V_{M}$.
\end{proposition}

We outline the decomposition strategy from \cite{liwski2025minimal} used to approach Question~\ref{quest 1}, presenting it in the form of an algorithm.

\begin{algorithm}[\textup{\cite[Section~4]{liwski2025minimal}}]\label{algo}
The strategy for decomposing $V_{\mathcal{C}(M)}$ follows these steps:
\begin{itemize}
\item[{\rm (i)}] We begin by identifying the minimal matroids $\min(M)$. 
\item[{\rm (ii)}] Next, the circuit variety is decomposed using Proposition~\ref{deco circ}. In the recursive decomposition step, this process is applied to each circuit variety that appears in the decomposition. The decomposition continues iteratively for any new circuit varieties encountered.
\item[{\rm (iii)}] If we reach to a circuit variety that coincides with its matroid variety then it is substituted by its matroid variety.
\item[{\rm (iv)}] After this step, we obtain a decomposition of $V_{\mathcal{C}(M)}$ as a union of matroid varieties. Then we determine the irreducibility of the matroids varieties in the decomposition.
\item[{\rm (v)}] Finally, redundant components~are~removed.
\end{itemize}
\end{algorithm}

\begin{remark}\label{grid}
    A specific hypergraph that arises in the study of conditional independence models with hidden random variables is a grid, denoted by $\Delta^{s,t}$, where $s$ and $t$ represent the number of states of certain hidden random variables. For this particular hypergraph, the associated hypergraph variety, along with the varieties that appear as irreducible components in its decomposition, have statistical significance. More precisely, they encode inference relationships among random variables and provide an analog of the intersection axiom in the presence of hidden random variables. These varieties and their decomposition have been studied in \cite{clarke2020conditional,clarke2022conditional,Fink} for specific parameters $s$ and $t$, but the most general case remains open. The algorithm and results of this paper provide a pathway toward addressing the general case. In particular, for the previously known cases, the components appearing in the decomposition are mostly nilpotent.
\end{remark}

\subsection{Connections to hypergraph varieties}

In this subsection, we examine how the families of nilpotent and solvable matroids relate to the questions introduced in the previous subsection, highlighting their connections and applications.

\medskip
\noindent{\bf Nilpotent matroids:} Our results on nilpotent matroids enhance the efficiency of Algorithm~\ref{algo}. Specifically, they refine step (iii) by allowing the replacement of a circuit variety associated with a nilpotent rank-three matroid, provided it has no points of degree greater than two, with its matroid variety, as ensured by Theorem~\ref{je}. For further details, we refer the reader to \cite{liwski2025minimal}. We illustrate this idea with an example from the same work, omitting certain details for brevity. A complete version of the example can be found in \cite{liwski2025minimal}.

\begin{example}[\textup{\cite[Subsection~5.1]{liwski2025minimal}}]
Consider the {\bf Fano plane} $M_{\text{Fano}}$ in Figure~\ref{new figure} (Right). The minimal matroids of $M_{\text{Fano}}$ are as follows; see Figure~\ref{figure min}, from left to right:

\begin{itemize}
\item[{\rm (i)}] The uniform matroid $U_{2,7}$.
\item[{\rm (ii)}] The matroids $M_{\text{Fano}}(i)$ for $i\in [7]$.
\item[{\rm (iii)}] A line of $M_{\text{Fano}}$, with the remaining four points coinciding outside this line. 
\item[{\rm (iv)}] A matroid with one line containing three parallel points and a free point outside it 
\end{itemize}

\begin{figure}
    \centering
    \includegraphics[width=0.8\textwidth, trim=0 0 0 0, clip]{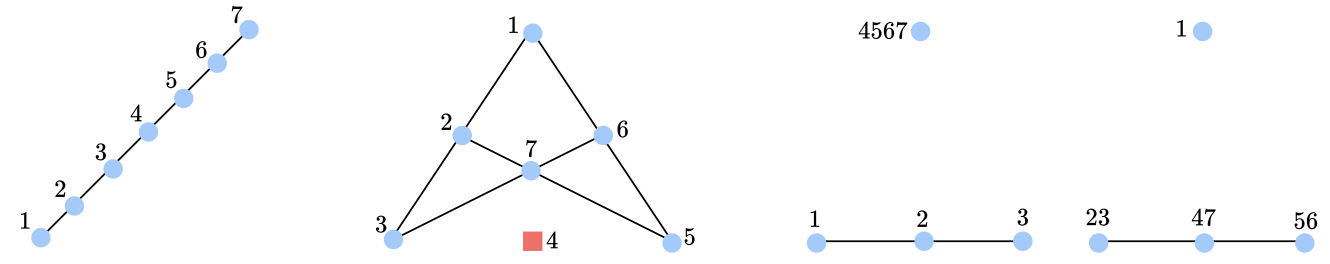}
    \caption{Minimal matroids of Fano configuration.}
    \label{figure min}
\end{figure}
We label the matroids of types (iii) and (iv) as
$A_{j},B_{k}$, 
for $j,k\in [7]$. Given that $M_{\text{Fano}}$ is not realizable, applying Proposition~\ref{deco circ} this leads to the following decomposition:

\[V_{\mathcal{C}(M_{\text{Fano}})}=V_{\mathcal{C}(U_{2,7})}\bigcup_{i=1}^{7}V_{\mathcal{C}(M_{\text{Fano}}(i))}\bigcup_{j=1}^{7}V_{\mathcal{C}(A_{j})}\bigcup_{k=1}^{7}V_{\mathcal{C}(B_{k})}.\]
Since the matroids $U_{2,7},A_{j}$ and $B_{k}$ are all nilpotent, their
matroid and circuit varieties coincide. Thus, 
\[V_{\mathcal{C}(M_{\text{Fano}})}=V_{U_{2,7}}\bigcup_{i=1}^{7}V_{M_{\text{Fano}}(i)}\bigcup_{j=1}^{7}V_{A_{j}}\bigcup_{k=1}^{7}V_{B_{k}}.\]
All matroids in this decomposition are solvable. Thus, 
all these varieties are irreducible. Therefore, the above decomposition is the irreducible decomposition of $V_{\mathcal{C}(M_{\text{Fano}})}$. 
\end{example}

\noindent{\bf Solvable matroids:} We now discuss the occurrence of solvable matroids in previous works.

\begin{itemize}
\item From Remark~\ref{arrangemnet}, we observe that the realization spaces of inductively connected line arrangements, as defined in \cite{nazir2012connectivity}, correspond to those of solvable rank-three matroids.
\item Identifying matroids with irreducible realization spaces is a crucial step in Algorithm~\ref{algo}, specifically in step (iv). Theorems~\ref{teosol} and~\ref{irreducible} establish the irreducibility of certain subfamilies of solvable matroids. For instance, this criterion was applied in the final step of the previous example to determine the irreducibility of the corresponding matroid varieties. For further examples, we refer the reader to \cite{liwski2025minimal, liwski2025}.
\end{itemize}

\noindent{\bf Paving matroids without points of degree greater than two:} The study of this family of matroids is closely related to the decomposition of grid hypergraphs from Remark~\ref{grid}. Specifically, setting $s=t=d$ in the grid $\Delta^{s,t}$ yields a variety $V_{\Delta^{s,t}}$ that coincides with the circuit variety of a $d$-paving matroid without points of degree greater than two. 
See \cite{Fatemeh4} for more details.

\vspace{-3mm}

\printbibliography 

\end{document}